\documentclass[12pt,a4paper]{amsart}

\usepackage{graphicx}
\usepackage{xcolor}
\usepackage{amsmath,amsfonts,amssymb,amsthm}
\usepackage{enumerate}
\usepackage{mathrsfs}
\usepackage{amscd}
\usepackage{euscript}
\usepackage{arydshln}
\usepackage[all]{xy}
\usepackage[
  a4paper,
  top=25mm,
  bottom=20mm,
  left=25mm,
  right=25mm
]{geometry}

\usepackage{hyperref}
\usepackage{cleveref}
\usepackage{xr-hyper}
\usepackage{doi}

\theoremstyle{definition}
\newtheorem{Def}{Definition}[section]

\newtheorem{Rm}{Remark}[section]

\theoremstyle{plain}
\newtheorem{Prop}[Def]{Proposition}
\newtheorem{Lem}[Def]{Lemma}
\newtheorem{Thm}[Def]{Theorem}

\numberwithin{equation}{section}

\newcommand{\authorfootnotesA}{\renewcommand\thefootnote{$\flat$}}
\newcommand{\authorfootnotesB}{\renewcommand\thefootnote{$\sharp$}}

\definecolor{light-gray}{gray}{0.5}
\begin{document}

\begin{center}
	\LARGE
    Factorization of Time-Ordered Exponentials for Wiener Space Transformations \par \bigskip

	\normalsize
	\authorfootnotesA
	Jir\^o Akahori\textsuperscript{1},
	\authorfootnotesB
    Takafumi Amaba\textsuperscript{2}\footnote{%
	This work was supported by JSPS KAKENHI Grant Number 22K03345.%
	},
    Tasuku Kubo\textsuperscript{1}

    \textsuperscript{1}Ritsumeikan University, 1-1-1 Nojihigashi, Kusatsu, Shiga, 525-8577, Japan

	\textsuperscript{2}Fukuoka University,
	8-19-1 Nanakuma, J\^onan-ku, Fukuoka, 814-0180, Japan.\par \bigskip

	\today
\end{center}

\begin{quote}{\small {\bf Abstract.}
We develop an operator-algebraic framework for change-of-variables formulas on Wiener space, interpreting them as arising from hidden symmetries acting on observables. We show that general transformations can be represented by time-ordered exponentials generated by annihilation and creation operators, and that these admit an explicit factorization into a determinant, a multiplication operator, and a translation operator. Taking expectations recovers the classical formulas, including the Ramer--Kusuoka formula.
}
\end{quote}

\renewcommand{\thefootnote}{\fnsymbol{footnote}}

\footnote[0]{email:\href {mailto:akahori@se.ritsumei.ac.jp}{akahori@se.ritsumei.ac.jp}(J.Akahori),\,%
\href {mailto:fmamaba@fukuoka-u.ac.jp}{fmamaba@fukuoka-u.ac.jp}(T.Amaba),
\href {mailto:ra0074vh@ed.ritsumei.ac.jp}{ra0074vh@ed.ritsumei.ac.jp}(T.Kubo)}

\footnote[0]{ 
2020 \textit{Mathematics Subject Classification}.
Primary
60H07; 
Secondary
17B65, 
81R10 
}

\footnote[0]{ 
\textit{Key words and phrases}.
Cameron--Martin--Maruyama--Girsanov formula;
Ramer--Kusuoka formula;
Heisenberg algebra;
CCR;
time-ordered exponential
}

\section{Introduction}

Classical change-of-variables formulas on Wiener space---such as the Cameron--Martin, Maruyama--Girsanov, and Ramer--Kusuoka formulas---are typically formulated and understood within a measure-theoretic framework. Starting from a more elementary formula, these expressions are obtained by extending the following change-of-variables formula on $\mathbb{R}^{d}$ to the path space:
\begin{equation}
\label{eq:general_CoV}
\begin{split}
&
\int_{\mathbb{R}^{d}} f(x) \,\mu_{d} (\mathrm{d}x) \\
&=
\int_{\mathbb{R}^{d}}
f(x-z(x))
\cdot
\det \left( I_{d} - \frac{\partial z}{\partial x}(x) \right)
\exp \left( \langle z(x), x \rangle_{\mathbb{R}^{d}} - \frac{ \Vert z(x) \Vert_{\mathbb{R}^{d}}^{2} }{2}  \right)
\,\mu_{d} (\mathrm{d}x).
\end{split}
\end{equation}
Here, $\mu_{d}(\mathrm{d}x) = (2\pi)^{-d/2} \exp\!\left(-\|x\|_{\mathbb{R}^{d}}^{2}/2\right)$ denotes the Gaussian measure, and $z = (z_{1}, z_{2}, \ldots, z_{d}) : \mathbb{R}^{d} \to \mathbb{R}^{d}$ is a suitable smooth transformation. The matrix $\frac{\partial z}{\partial x} = \left( \frac{\partial z_{i}}{\partial x_{j}} \right)_{1 \leq i,j \leq d}$ denotes the Jacobian matrix of $z$. Strictly speaking, the determinant in the above formula should be taken in absolute value. However, since we are interested only in the algebraic structure underlying this formula, we assume that $z$ is sufficiently small so that $\det \left( I_{d} - \frac{\partial z}{\partial x}(x) \right) > 0$.
On the path space $W = \{ w \in C([0,1] \to \mathbb{R}): w(0)=0 \}$, there exists, at least formally, a probability measure with the following heuristic expression, called the \emph{Wiener measure}:
\begin{equation*}
\mu (\mathrm{d}w)
\propto
\exp \left( - \frac{1}{2} \int_{0}^{1} (\, \text{`}\dot{w}(t) \text{'} \,)^{2} \,\mathrm{d}t \right)
\, \text{`} \mathrm{d}w \text{'}.
\end{equation*}
The connection with $\mathbb{R}^{d}$ is given by the pushforward via the map $W \ni w \mapsto x = (\Delta w_{1}, \Delta w_{2}, \ldots, \Delta w_{d}) \in \mathbb{R}^{d}$, where $[0,1]$ is divided into $d$ subintervals and $\Delta w_{i}$ denotes the increment over each subinterval, together with an appropriate scalling. While these formulas provide explicit Radon--Nikod{\'y}m densities describing how Wiener measure transforms under suitable shifts, their structural origin has remained largely analytic. We interpret this structure as arising from hidden symmetries acting on observables.

In this paper, we propose a fundamentally different perspective: we interpret these change-of-variables formulas as arising from \emph{hidden symmetries acting on observables}, and we develop an operator-algebraic framework in which these symmetries become explicit. Our main result (Theorem~\ref{Thm:main}) shows that the transformation associated with a general (possibly anticipating) shift can be realized as a \emph{time-ordered exponential} generated by a system of annihilation and creation type operators, and that this operator admits a striking factorization into three components:
$
\text{(time-ordered exponential)}
= \text{(determinant)} \times \text{(multiplication operator)} \times \text{(translation operator)}.
$
This factorization may be regarded as a noncommutative analogue of the classical Jacobian formula, and provides a unified structural explanation for the appearance of determinant and exponential terms in Wiener space transformations.

To extract the algebraic structure underlying the formula, we begin by recalling the classical change-of-variables formulas from which it originates.

\subsection{The Cameron--Martin formula}

It is instructive to begin with the case where $z(x) \equiv h \in \mathbb{R}^{d}$. In this situation, the determinant appearing in the equation (\ref{eq:general_CoV}) is equal to $1$, and the Radon--Nikod\'{y}m density takes the form
$
\exp \left(
- \langle h, x \rangle_{\mathbb{R}^{d}}
- \frac{\|h\|_{\mathbb{R}^{d}}^{2}}{2}
\right)
$.
As the corresponding expression lifted to Wiener space via the projective limit, one obtains the following formula, known as the \emph{Cameron--Martin formula}:
\begin{equation}
\label{eq:CM_formula}
\int_{W} f(w) \,\mu (\mathrm{d}w)
=
\int_{W}
f(w-h)
\cdot
\exp \left( \int_{0}^{1} \dot{h}(s) \,\mathrm{d}w(s)
-
\frac{1}{2}
\int_{0}^{1} \dot{h}(s)^{2} \,\mathrm{d}s
\right)
\,\mu (\mathrm{d}w).
\end{equation}
Here, $h$ is an absolutely continuous function whose derivative is square-integrable, and $\dot{h}$ denotes its derivative. The subspace $H$ consisting of all such absolutely continuous functions with square-integrable derivatives is called the \emph{Cameron--Martin subspace}.

From the viewpoint of functional analysis, $W$ is a separable Banach space under the uniform norm, while $H$ becomes a Hilbert space equipped with the $L_{2}$-inner product of derivatives. In this setting, taking the dual of the dense embedding $H \hookrightarrow W$, one obtains $W^{*} \hookrightarrow H^{*} \cong H$, whose image is dense in $H$. Consequently, every element $h \in H$ can be approximated by elements of $W^{*}$, and by regarding this approximation as a convergent sequence in $L_{2}(W,\mu)$ (which contains $W^{*}$), one obtains an embedding, called the \emph{stochastic extension}, $H \ni h \mapsto \widetilde{h} \in L_{2}(W,\mu)$.

The expression $\int_{0}^{1} \dot{h}(s)\,\mathrm{d}w(s)$ is one representation of $\widetilde{h}(w)$, based on the structure of $W$ as a path space, and is called the \emph{Wiener integral}. This construction can in fact be carried out in a general separable Banach space equipped with a Gaussian measure, namely, in an \emph{abstract Wiener space} developed by L.~Gross (\cite{Gro67}).

The formula (\ref{eq:CM_formula}) is due to Cameron--Martin (\cite{Cam}), who, at a time when the importance of integration with respect to Wiener measure was becoming increasingly apparent from various perspectives, clarified precisely how such integrals transform under translations of the integration variable. That is, they identified the exact form of the weight (Radon--Nikod\'{y}m density) with which the transformed integral can be expressed again as an integral with respect to the original Wiener measure, as in the formula above. It is worth noting that they referred to integration with respect to Wiener measure as the Wiener integral. This result forms the foundation of the vast field later known as \emph{Malliavin calculus}.

\subsection{The Maruyama--Girsanov formula}

Next, consider the case where $\frac{\partial z^{i}}{\partial x^{j}} = 0$ for all $j < i$, that is, when each $z^{i}(x)$ depends only on $x_{1}, x_{2}, \ldots, x_{i-1}$.
Following the earlier viewpoint of interpreting a point in $\mathbb{R}^{d}$ as a `list of increments of a path,' this condition means that the next increment $z^{i}(x)$ does not depend on future components, but only on the past history $x_{1}, x_{2}, \ldots, x_{i-1}$. It is therefore natural to recognize this as an assumption suited to the treatment of filtrations in stochastic analysis.

An analogous condition exists on Wiener space, where it is referred to as predictability. Owing to the specific structure of Wiener space, this is essentially equivalent to adaptedness (more precisely, progressive measurability).
In this case, replacing $z(x)$ by $Z(w) = \left( \int_{0}^{t} z_{s}(w)\,\mathrm{d}s \right)_{0 \leq t \leq 1}$, the corresponding change-of-variables formula on Wiener space can be written as follows, and is known as the (\emph{Cameron--Martin--})\emph{Maruyama--Girsanov formula}:
\begin{equation}
\label{eq:MG_formula}
\begin{split}
\int_{W} f(w) \,\mu (\mathrm{d}w)
=
\int_{W}
f(w-Z(w))
\cdot
\exp \left( \int_{0}^{1} z_{t}(w) \,\mathrm{d}w(t)
-
\frac{1}{2}
\int_{0}^{1} ( z_{t}(w) )^{2} \,\mathrm{d}t
\right)
\,\mu (\mathrm{d}w).
\end{split}
\end{equation}

During the 1940s and 1950s, against the background of the theory of Markov processes developed by A.~N.~Kolmogorov and W.~Feller and the development of stochastic integration by K.~It\^{o}, the relationship between stochastic processes described by stochastic differential equations and Markov processes came to be understood in a systematic way.
In this context, G.~Maruyama (\cite{Mar55}) constructed diffusion processes as solutions to stochastic differential equations obtained by adding a drift term to the Wiener process. Moreover, this construction can be realized via difference approximations (now the so-called \emph{Euler--Maruyama scheme}), and their convergence and distributional properties were studied.
Subsequently, I.~V.~Girsanov (\cite{Gir60}) showed that such drifted processes can be described as absolutely continuous changes of measure with respect to the Wiener measure, thereby reinterpreting these structures within the framework of measure transformations.

This result (\ref{eq:MG_formula}) yields a wide range of properties related to Brownian motion, including methods for solving stochastic differential equations and estimates of the densities of their solutions.

\subsection{The Ramer--Kusuoka formula}

When $z$ is not predictable (i.e., anticipating), the Radon--Nikod\'{y}m density on Wiener space is given by the following expression under suitable conditions.
\begin{equation}
\label{eq:RK_formula}
\begin{split}
&
\int_{W} f(w) \,\mu (\mathrm{d}w) \\
&=
\int_{W}
f(w-Z(w))
\cdot
\mathrm{det}_{2} ( I - D Z (w) )
\exp \left( \int_{0}^{1} z_{t}(w) \,\delta w(t)
-
\frac{1}{2}
\int_{0}^{1} ( z_{t}(w) )^{2} \,\mathrm{d}t
\right)
\,\mu (\mathrm{d}w)
\end{split}
\end{equation}
(Compare with (\ref{eq:general_CoV}).)
Here, $Z(w) = \left( \int_{0}^{t} z_{s}(w)\,\mathrm{d}s \right)_{0 \leq t \leq 1}$, and $DZ$ denotes the Malliavin derivative, which corresponds to the Jacobian of the map $Z : W \ni w \mapsto Z(w) \in H$
($\subset W$).
The symbol $\det_{2}$ denotes the Carleman--Fredholm determinant, an infinite-dimensional analogue of the Jacobian, and $\int_{0}^{1} z_{t}(w)\,\delta w(t)$ is the Skorokhod integral, which extends the It\^{o} integral above to anticipating integrands.
Transformations of this type were already considered by Cameron--Martin (\cite{Cam2, Cam3}); however, an explicit formula of the above form was first derived Ramer (\cite{Ram}) and later refined by S.~Kusuoka (\cite{Kus}).
As is pointed in \cite{ZZ}, if $DZ$ is {\em quasi-nilpotent}, meaning that $\mathrm{Tr} (DZ)^{n} = 0$ for all $n$, the density (\ref{eq:RK_formula}) reduces to the one for Girsanov--Maruyama.
Subsequently, Bukdahn~\cite{Buc2} studied anticipating Girsanov transformations in which the adaptedness requirement is removed, and attempts have been made to express the Carleman--Fredholm determinant appearing in the Ramer--Kusuoka formula in the form
$
\exp \int (\text{trace-type quantity around time $t$}) \,\mathrm{d}t
$.

\hrulefill

Whether on $\mathbb{R}^{d}$ or on $W$, the left-hand side of (\ref{eq:general_CoV}), (\ref{eq:CM_formula}), (\ref{eq:MG_formula}) or (\ref{eq:RK_formula}) represents the expectation (integral) of $f$ itself, whereas the right-hand side yields the same quantity as the left-hand side by first modifying $f$ in some manner and then taking its expectation.

From this perspective, one may regard the left-hand side as expressing that some hidden symmetry acts on the function $f$ in such a way that, after taking expectation, the integral of $f$ itself reappears. On the other hand, the right-hand side can be interpreted as giving an explicit description of how this symmetry acts on $f$.

Such a reinterpretation may be viewed as a natural attempt to uncover an underlying algebraic structure in probabilistic objects. It is then natural to expect that this symmetry is realized as an element of some large group, that is, something that can be expressed in the form sort of $\exp(...)$.

\emph{The main objective of this paper is to describe the operator corresponding to the Ramer--Kusuoka type change of variables as a time-ordered exponential constructed from abstract annihilation and creation type operators, and to show that it admits a factorization into the product of two multiplication operators, involving determinant factors, and a translation operator.}

Our approach developed below shows that there exists a unified algebraic symmetry underlying these transformations, which is not visible at the level of expectations but becomes manifest only at the level of operators.

\subsection{A generalization of CCR}
In view of the fact that this action becomes invisible after taking expectation on the left-hand side of (\ref{eq:general_CoV}), (\ref{eq:CM_formula}), (\ref{eq:MG_formula}) or (\ref{eq:RK_formula}), it is natural to consider that this \emph{hidden symmetry} should be expressed in the form $1 + (\text{adjoint of an annihilator})$, or more typically as something of the form $\exp(\text{adjoint of an annihilator})$. Here, `adjoint' refers to the adjoint operator with respect to the $L^{2}(\mu)$-inner product associated with the measure $\mu$.

This action is expected to be determined by some quantity related to $z$. To identify its infinitesimal generator, we consider the above change-of-variables formula by regarding $z$ as a perturbation of $x$ (or $w$), and taking (functional) derivatives. Then, upon fixing a basis $(h_{i})_{i}$ of the underlying space, one is naturally led to the derivative $D_{i}$ in the $i$-th coordinate direction and its $L^{2}(\mu)$-adjoint $D_{i}^{*}$.

More precisely, in the change-of-variables formula above, replacing $f$ by a product $f \cdot g$, and setting $z(x) \equiv \varepsilon \cdot h_{i}$, we apply $\left. \frac{\mathrm{d}}{\mathrm{d}\varepsilon} \right|_{\varepsilon = 0}$ to both sides. Using the identity $D_{i}^{*} f(x) = - D_{i} f(x) + x_{i} f(x)$, we obtain $\langle D_{i} f, g \rangle_{L^{2}(\mu)}
=
\langle f, D_{i}^{*} g \rangle_{L^{2}(\mu)}$.

This relation holds not only when $\mu$ is a Gaussian measure, but more generally when $\mu$ is described by a smooth probability density function $p(x) = \frac{\mu(\mathrm{d}x)}{\mathrm{d}x}$, by interpreting $D_{i}^{*} f = - D_{i} f + (-D_{i} \log p)\, f$. From this definition, one has $D_{i} + D_{i}^{*} = -D_{i} \log p$, which is, in particular, a multiplication operator. The operators $D_{i}^{*}$ obtained in this way satisfy the following commutation relations with the operators $D_{i}$:
\begin{equation*}
( D_{i} D_{j}^{*} - D_{j}^{*} D_{i} )f(x)
=
(- D_{i}D_{j} \log p(x) ) \cdot f(x)
\end{equation*}
In other words, $[ D_{i}, D_{j}^{*} ] = D_{i} D_{j}^{*} - D_{j}^{*} D_{i} = - D_{i} D_{j} \log p$ is a multiplication operator. In particular, when $\mu$ is a Gaussian measure, this reduces to the well-known \emph{canonical commutation relations} (CCR):
$[ D_{i}, D_{j}^{*} ] = \delta_{i,j}$. Moreover, it is immediate to verify that both $[ D_{i}, f ] = \frac{\partial f}{\partial h_{i}}
\quad \text{and} \quad
[ D_{i}^{*}, f ] = - \frac{\partial f}{\partial h_{i}}$ are also multiplication operators.

To abstract the probabilistic structure described above, we now restart from the following abstract framework. Let $k=\mathbb{R}$ and $R$ be a commutative unital algebra over $k$, which we identify with the collection of multiplication operators by elements of $R$.

Let $\mathrm{End}_{k}(R)$ be the algebra of $k$-linear operators on $R$, and denote the commutator in $\mathrm{End}(R)$ by $[A,B] = AB - BA$. We denote by $\mathrm{Diff}(R)$ the ring of differential operators on $R$.
The action of $\mathrm{End}_{k}(R)$ on $R$ is denoted by $R \ni f \mapsto A.f = A(f) \in R$ for $A \in \mathrm{End}_{k}(R)$.

Suppose that we are given operators $D_{i}$, $D_{i}^{*} \in \mathrm{End}_{k}(R)$, for $i \in \mathbb{N}$, satisfying the following conditions:
\begin{itemize}
\item[1.]
$[D_{i}, D_{j}] = [D_{i}^{*}, D_{j}^{*}] = 0$,
\item[2.]
$D_{i} + D_{i}^{*} \in R$,
\item[3.]
$[D_{i}, R]$, $[D_{i}^{*}, R] \subset R$, equivalently, $[[D_{i}, R], R] = [[D_{i}^{*}, R], R] = 0$, that is, $D_{i}$ and $D_{i}^{*}$ are differential operators of order at most one.
\end{itemize}
From the above properties, we have $[ D_{i}, D_{j}^{*} ] = [ D_{i}, D_{j} + D_{j}^{*} ]
\subset [ D_{i}, R ]
\subset R$, and hence $[ D_{i}, D_{j}^{*} ]$ is a multiplication operator. In this sense, the present framework can be regarded as a generalization of CCR.
Also, note that, since $[\, D_{i} + D_{i}^{*}, f \,] = 0$ for every $f \in R$, it follows in particular that $[\, D_{i}^{*}, f \,] = -[\, D_{i}, f \,]$.

We denote by $\mathcal{D}(R)$ the subalgebra of $\mathrm{Diff}(R)$ generated by $\{ D_{i}, D_{i}^{*} \}_{i}$ together with $R$. Here, we think of $D_{i}$ as annihilators, and have in mind situations in which $D_{i}^{*}$ can be realized as adjoint operators in $L_{2}$ with respect to some measure.

Let $\mathcal{I} = ( D_{1}^{*}, D_{2}^{*}, D_{3}^{*}, \ldots ) \, \mathcal{D}(R)$ be the right ideal of $\mathcal{D}(R)$ consisting of elements that admit a representation in which the operators $D_{i}^{*}$ appear on the far left. Then elements of $\mathcal{I}$ are to be regarded as `ajoints' of annihilators.

\subsection{Time-ordered exponential as symmetry}

Fix a family $Z = \{ Z_{i} \}_{i \in \mathbb{N}} \subset R$, and assume that $Z_{i} = 0$ for all but finitely many indices $i$. Then the $\mathbb{N} \times \mathbb{N}$ matrices $\Phi_{Z} = \Phi$ and $\Psi_{Z} = \Psi$, defined below, have the property that all their entries vanish except for finitely many columns.
\begin{equation*}
\Phi
= (\Phi_{i,j})_{i,j \in \mathbb{N}}
= ( D_{i}^{*} Z_{j} )_{i,j \in \mathbb{N}},
\quad
\Psi
= (\Psi_{i,j})_{i,j \in \mathbb{N}}
= ( [ D_{i}, Z_{j} ] )_{i,j \in \mathbb{N}}
\end{equation*}

In general, consider the collection of matrices whose entries vanish except for finitely many columns:
$\bigcup_{n=1}^{\infty}
\mathrm{Hom}_{k} \left( \mathcal{D}(R)^{n}, \mathcal{D}(R)^{\mathbb{N}} \right)
$,
where the union is taken under the natural embeddings
$
\mathrm{Hom}_{k} \left( \mathcal{D}(R)^{n}, \mathcal{D}(R)^{\mathbb{N}} \right)
\hookrightarrow
\mathrm{Hom}_{k} \left( \mathcal{D}(R)^{\mathbb{N}}, \mathcal{D}(R)^{\mathbb{N}} \right)
$.
Note that this collection is closed under the usual matrix multiplication. Moreover, for such matrices, the trace $\mathrm{Tr}$ is well defined as a finite sum, and the determinant can also be defined without difficulty by the usual formula.

As mentioned above, we are interested in symmetries that can be expressed in the form $\exp(\text{adjoint of an annihilator})$. In the present setting, however, the algebra under consideration is noncommutative, and therefore several notions of an exponential function may arise. Among these, we adopt the time-ordered version.

To this end, we adjoin an indeterminate $t$, commuting with $\mathcal{D}(R)$, and consider the formal power series algebra $\mathcal{D}(R)[\![ t ]\!]$. On this space, we study the solution of the following differential equation.
\begin{equation}
\label{eq:dfe}
\frac{\mathrm{d}}{\mathrm{d} t} S(t)
=
\left( \sum_{n=0}^{\infty} [\,\mathrm{Tr} ( \Phi_{Z} (- \Psi_{Z} )^{n} ) \,] \,t^{n} \right)
S(t),
\quad S(0) = 1
\end{equation}
The solution $S(t)$ is typically expressed as
\begin{equation*}
\mathcal{T}\!\exp \left(
\int_{0}^{t}
\sum_{n=0}^{\infty}
\left[
\mathrm{Tr}\!\left( \Phi_{Z} \, (-\Psi_{Z})^{n} \right)
\right]
(t^{\prime})^{n}
\mathrm{d}t^{\prime}
\right)
=
\mathcal{T}\!\exp \left(
\sum_{n=0}^{\infty}
\frac{t^{\,n+1}}{n+1}
\left[
\mathrm{Tr}\!\left( \Phi_{Z} \, (-\Psi_{Z})^{n} \right)
\right]
\right) ,
\end{equation*}
and is formulated as an element of
$
\exp\big( \mathcal{I}[\![ t ]\!] \big) \subset \mathcal{D}(R)[\![ t ]\!]
$
(\cite{GKLLRT95}).
Such solutions of differential equations are referred to as \emph{time-ordered exponentials}, and arise naturally in contexts involving hidden symmetries induced by gauge transformations, such as diffusion semigroups viewed as observables with harmonic functions as states, and Wilson lines.

In F.~J.~Dyson~\cite[Equation~(29)]{Dys49}, the product is introduced as a path-dependent ordered product $P$, corresponding to a path passing through specified points in a prescribed order. This notation is standard among experts in gauge theory, particularly in the computation of Wilson loops. In our context, however, we are not concerned with highly variable paths in operator space, but rather focus essentially on a single type of path. Therefore, instead of emphasizing path dependence, we adopt the notation $\mathcal{T}$, which highlights dependence on time ordering.

For reference, in R.~P.~Feynman~\cite{Fey51}, no special notation is introduced for time-ordered products; as a result, the time-ordered exponential is simply written as $\exp$. G.-C.~Wick (\cite{Wic50}) introduced his own time-ordering operator $T$, but it differs from the modern standard definition of time ordering in the treatment of sign factors.

A direct computation of such time-ordered exponentials is scarcely known outside the commutative case; in most situations, they admit only a formal representation as a series of multiple integrals. We aim to compute them more explicitly and to describe them in simpler terms. In doing so, we show that this corresponds to evaluating the change-of-variables formula at the level of observables, which, upon taking expectation (integration), reduces to the well-known change-of-variables formula. To describe the result of this computation of the time-ordered exponential, we now introduce the appropriate terminology.

\subsection{Translation operator, determinant and normal order product}
\label{Sec:det_order}

Fix a family $Z = \{ Z_{i} \}_{i \in \mathbb{N}} \subset R$ such that $Z_{i} = 0$ for all but finitely many indices $i$. We then define operators $D_{Z}, D_{Z}^{*} \in \mathcal{D}(R)$ by the following formulas.
\begin{equation*}
D_{Z} = \sum_{i} Z_{i} D_{i},
\quad
D_{Z}^{*} = \sum_{i} Z_{i} D_{i}^{*}
\end{equation*}

We define the translation operator $T_{Z}(t) \in \mathcal{D}(R)[\![ t ]\!]$ by the following formula.
\begin{equation*}
T_{Z} (t)
=
\sum_{n=0}^{\infty} \frac{t^{n}}{n!}
\sum_{i_{1}, i_{2}, \ldots , i_{n}}
Z_{i_{1}} Z_{i_{2}} \cdots Z_{i_{n}}
D_{i_{1}} D_{i_{2}} \cdots D_{i_{n}}
\end{equation*}

To gain some intuition for this operator, consider, for example, its action on a polynomial function $f(x)$ on $\mathbb{R}$. In this case, we take a representation given by $Z_{1} = h \in \mathbb{R}$, $Z_{2} = Z_{3} = \cdots = 0$, $D_{1} = \frac{\mathrm{d}}{\mathrm{d}x}$, $D_{2} = D_{3} = \cdots = 0$, and $t \in \mathbb{R}$. Then we have
\begin{equation*}
f(x + t \cdot h)
=
\sum_{n=0}^{\infty} \frac{ (t h )^{n}}{n!}
\frac{\mathrm{d}^{n}f(x)}{\mathrm{d}x^{n}}
=
\sum_{n=0}^{\infty} \frac{t^{n}}{n!}
\,
\left( h \frac{\mathrm{d}}{\mathrm{d}x} \right)^{n}
f(x)
\Bigg(
=
\big( \mathrm{e}^{t\, h D_{1}} f \big)(x)
\Bigg)
=
[ T_{Z}(t) f ](x).
\end{equation*}
This observation provides a useful way to develop an intuitive understanding of the operator $T_{Z}(t)$.

From this point on, we introduce a unary operation $:\bullet:$ on $\mathcal{D}(R)$, defined by the following rule, which is analogous to the notion known as the normal-ordered  (\cite{MJD}).

(i) For $a_{n} \in \mathcal{D}(R)$, $n \in \mathbb{N}$, we define $:\!\sum a_{n}\!:\; =\; \sum :a_{n}:$.
(ii) For monomials $a$ and $b$, we set $: a D_{i} b : \;=\; :ab:\, D_{i}$ and $: a Z_{i} b : \;=\; Z_{i} \, :ab:$. Hence, for example, the following relations hold.
\begin{equation*}
: \left( \sum_{i} Z_{i} D_{i} \right)^{2} \left( \sum_{i} Z_{i} D_{i}^{*} Z_{i} \right)^{2} :
\; = \;
\sum_{i,j,k,l}
Z_{i} Z_{j} Z_{k} Z_{l}
D_{k}^{*} D_{l}^{*} D_{i} D_{j}
\end{equation*}
This unary operation extends in an obvious manner to a unary operation on $\mathcal{D}(R)[\![t]\!]$ so as to preserve the formal parameter $t$. Under this notation, the translation operator introduced above can be expressed as $T_{Z}(t) \;=\; : \mathrm{e}^{\, t \, D_{Z} } :$. Moreover, although $D_{i} + D_{i}^{*}$ is a multiplication operator, it is not necessarily expressible as a linear combination of the $Z_{j}$'s. Hence, although $:D_{i} + D_{i}^{*}:\;=\; D_{i} + D_{i}^{*}$ holds, in general one cannot expect that $:( D_{i} + D_{i}^{*} )^{n}: \;=\; ( D_{i} + D_{i}^{*} )^{n}$ holds.

For an element $A = (A_{i,j})_{i,j \in \mathbb{N}}$ of the collection of matrices whose entries vanish except for finitely many columns, $\bigcup_{n=1}^{\infty}
\mathrm{Hom}_{k}\left( \mathcal{D}(R)^{n}, \mathcal{D}(R)^{\mathbb{N}} \right)$, and for a multi-index $I = (i_{1}, i_{2}, \ldots, i_{n}) \in \mathbb{N}^{n}$, we define the minor $\det(A \mid I)$ of $A$ by the following formula.
\begin{equation*}
\det ( A \mid I )
=
\sum_{\sigma \in \mathfrak{S}_{n}}
\mathrm{sgn} ( \sigma )
A_{i_{1}, i_{\sigma (1)}}
A_{i_{2}, i_{\sigma (2)}}
\cdots
A_{i_{n}, i_{\sigma (n)}}
\end{equation*}
In this setting, we define $\det(1 + tA) \in \mathcal{D}(R)[\![t]\!]$ by the following formula.
\begin{equation}
\label{eq:det}
\det ( 1 + t A )
=
\sum_{n=0}^{\infty}
\frac{ t^{n} }{ n! }
\sum_{ I \in \mathbb{N}^{n} }
\det ( A \mid I )
\end{equation}
Here, in the case where $I \in \mathbb{N}^{0}$, that is, $I = ()$ (empty tuple), we conventionally set $\det (\, A \mid () \,) = 1$.
If all the entries of $A$ commute, then the usual algebraic properties of the determinant are preserved.

Finally, we define $\mathcal{E}_{Z}(t) \in \mathcal{D}(R)[\![t]\!]$ by the following formula.
\begin{equation*}
\mathcal{E}_{Z}(t)
\, = \,
:
\exp \left(\,
t ( D_{Z} + D_{Z}^{*} )
\,\right)
:
\end{equation*}
In fact, this is defined as an element of $R[\![t]\!]$, that is, a formal power series whose coefficients are multiplication operators.
Actually, we have 
\begin{equation*}
\begin{split}
: (D_{i_1}+D^*_{i_1}) \cdots (D_{i_n}+D^*_{i_n}) :
\;&=\;
D^*_{i_{1}}
:(D_{i_2}+D^*_{i_2}) \cdots (D_{i_n}+D^*_{i_n}): \\
&\hspace{9mm}+ \;
: (D_{i_2}+D^*_{i_2}) \cdots (D_{i_n}+D^*_{i_n}) :D_{i_{1}}.
\end{split}
\end{equation*}
Since we have 
$
D_{i}^{*} R + R D_{i}
=
[ D_{i}^{*},  R ] + R ( D_{i}^{*}+D_{i}) \subset R
$,
we get
$
: (D_{i_1}+D^*_{i_1}) \cdots (D_{i_n}+D^*_{i_n}) : \in R
$
inductively.
Thus $\mathcal{E}_{Z}(t)$ is a multiplication operator.

We are now ready to state our main result. It shows that transformations which have hitherto been treated only in measure-theoretic terms can instead be handled within the framework of operator calculus. Our main theorem is the following identity in $\mathcal{D}(R)[\![t]\!]$.

\medskip

\begin{Thm} 
\label{Thm:main}
We have
$\displaystyle
\mathcal{T}\!\exp \left( \sum_{n=0}^{\infty} [\, \mathrm{Tr} ( \Phi_{Z} ( -\Psi_{Z} )^{n} ) \,] \, \frac{ t^{n+1} }{ n+1 } \right)
=
\det ( 1 + t \Psi_{Z} )
\mathcal{E}_{Z}(t)
T_{-Z}(t)
$.
\end{Thm} 

In other words, the time-ordered exponential appearing on the left-hand side can be computed as the application of the translation operator followed by two multiplication operators.

The above result may be regarded as a noncommutative version of the following identity, which holds for finite-dimensional matrices $A$ and $B$ over a scalar field, with $B$ invertible.
\begin{equation*}
\exp \left(
\sum_{n=0}^{\infty}
\frac{t^{n+1}}{n+1}
\mathrm{Tr} ( A B^{n} )
\right)
=
\det \left(
\exp \left(
- AB^{-1} \log ( I - t B )
\right)
\right)
\end{equation*}
In particular, in the case $A=B$, the following well-known identity holds even without assuming invertibility:
\begin{equation*}
\exp \left(
\sum_{n=0}^{\infty}
\frac{t^{n+1}}{n+1}
\mathrm{Tr} ( A^{n+1} )
\right)
=
\det ( I - t A )^{-1}.
\end{equation*}

In our framework, the classical change-of-variables formula described above is realized by letting the operator appearing in Theorem 1.1 act on the `vacuum' $1$, and then taking the expectation (i.e., integration). Therefore, apart from the determinant factor $\det ( 1 + t \Psi_{Z}$, the multiplicative operator $\mathcal{E}_{Z}(t)$ should also be regarded as the part that describes the Radon--Nikod\'{y}m density of this transformation, in particular its positivity.

For $f(t) \in R[\![t]\!]$, the exponential $\exp ( f(t) )$ is defined by the usual formal Taylor series. When realized, for example, as a function in $C([0,1] \to \mathbb{R})$, it becomes a positive function. (On the other hand, $\det ( 1 + t \Psi_{Z} )$ should also be positive for sufficiently small $t$.)

Thus, the following theorem can be a basis to prove the absolute continuity of our `translation.'

\medskip

\begin{Thm} 
\label{Thm:MB}
Suppose that $D_{i} 1 = 0$. Then,
$\displaystyle
\mathcal{E}_{Z}(t).1
= 
\exp \left(
\int_{0}^{t}
:\mathrm{e}^{ -t^{\prime} D_{Z} } D_{Z}^{*}: (1)
\, \mathrm{d}t^{\prime}
\right)
$.
\end{Thm}  

In light of this result, we may regard it as an abstract characterization of the absolute continuity of `translations.'

\subsection{The organization of the paper}

We first provide a proof of the main theorem in the next section. We then realize our algebra $\mathcal{D}(R)$ on the space generated by Hermite polynomials over the classical Wiener space, and explain how the above identity yields the Ramer--Kusuoka formula upon taking expectation. Furthermore, we show that this classical formula admits extensions to curved Wiener spaces and to McKean--Vlasov equations.

\subsection{Historical Remarks}
The study of absolute continuity of measures under transformations has a long history in stochastic analysis. 
The foundational work was laid by \cite{Cam, Cam2, Cam3} by Cameron and Martin, who characterized the density for linear translations in the Wiener space. 
This was later extended to more general settings by Gross \cite{Gro}, Segal \cite{Seg}, and Kuo \cite{Kuo}.

A significant breakthrough for non-linear and non-anticipative transformations was achieved by Ramer \cite{Ram} and Kusuoka \cite{Kus}, leading to the Ramer-Kusuoka formula. 
Their work introduced the use of the Carleman-Fredholm determinant to handle the ``Jacobian" term arising from the non-linear shift. 
Subsequently, the formula was further developed and refined in various directions. 
Bukdahn \cite{Buc2} explored anticipating Girsanov transformations without the requirement of adaptedness, while the case of infinite dimensional spaces was studied by authors including J. Zabczyk.

In this paper, we revisit these classical results from a purely algebraic perspective. 
By utilizing the action of a generalized Heisenberg algebra and normal order products, we provide a unified framework that recovers the classical formulas on Wiener spaces.

\section{Proofs of Theorems~\ref{Thm:main} and \ref{Thm:MB}}
\subsection{Proof of Theorem~\ref{Thm:main}}

Let $s_n$, $n =0,1,2,\ldots$ are non-commutative version of the \emph{complete homogeneous symmetric functions}
defined inductively by
$s_0 \equiv 1$ and 
\begin{equation*}
s_{n+1} (x_0, x_1,\cdots,x_{n}) = \frac{1}{n+1} 
\sum_{k=0}^n x_k s_{n-k}. 
\end{equation*} 
Here $x_{i}$'s are (not necessarily commutative) indeterminates.

By the following lemma, the solution 
of the differential equation (\ref{eq:dfe}) can be expressed as a power series in $t$,
however, this is nothing but a reformulation, in the present context, of the content of \cite[Definition~3.1, Proposition~3.3]{GKLLRT95}.
(Note that in \cite{GKLLRT95}, a version is considered in which the order of multiplication of $b(t)$ and $a(t)$ on the right-hand side of the differential equation below is reversed. By reinterpreting the formula with the order of multiplication reversed, the same expression is obtained.)

\medskip

\begin{Lem} 
\label{Lem:UNQ}
For
$b(t) = \sum_{n=0}^{\infty} t^n b_n \in \mathcal{D}(R)[\![t]\!]$,
the differential equation in $\mathcal{D}(R)[\![t]\!]$,
\begin{equation}
\label{eq:NCDE}
a^{\prime}(t) = b(t) a (t)
\quad
\text{with}
\quad
a(0) = 1,
\end{equation}
have the unique solution 
\begin{equation*}
a(t) = \sum_{n=0}^\infty t^n s_n (b_0,\cdots, b_{n-1}).
\end{equation*}
\end{Lem} 
\begin{proof} 
We set $a(t) = \sum_{n=0}^\infty t^n a_n $. 
Then, by (\ref{eq:NCDE}), we have
\begin{equation*}
\sum_{n=0}^\infty t^n (n+1) a_{n+1}
= a^{\prime}(t)
= b(t) a(t)
= \left( \sum_{n=0}^\infty t^n b_n \right) \left( \sum_{n=0}^\infty t^n a_n \right)
=
\sum_{n=0}^\infty t^n \sum_{k=0}^n b_k a_{n-k},
\end{equation*}
so that we get
$
a_{n+1} =\frac{1}{n+1} \sum_{k=0}^n b_k a_{n-k}
$.
Since $a_0 = 1$, from the definition of $s_n$ we obtain $a_n=s_n (b_0,\cdots, b_{n-1})$ for all $n \geq 0$. 
\end{proof} 

Now, by (\ref{eq:dfe}), the expression
$\displaystyle
S(t)
=
\mathcal{T}\!\exp \left(
\sum_{n=0}^{\infty}
\frac{t^{\,n+1}}{n+1}
\left[\,
\mathrm{Tr}\!\left( \Phi_{Z} \, (-\Psi_{Z})^{n} \right)
\,\right]
\right)
$
satisfies the differential equation (\ref{eq:NCDE}) with
$
b(t)
=
\sum_{n=0}^{\infty}
t^{n} \, [\, \mathrm{Tr}( \Phi_{Z} (-\Psi_{Z})^{n} ) \,]
$,
and hence, by Lemma~\ref{Lem:UNQ}, we obtain the following expression.
\begin{equation*}
\begin{split}
&
\mathcal{T}\!\exp \left(
\sum_{n=0}^{\infty}
\frac{t^{\,n+1}}{n+1}
\left[\,
\mathrm{Tr}\!\left( \Phi_{Z} \, (-\Psi_{Z})^{n} \right)
\,\right]
\right) \\
&=
1
+
\sum_{n=1}^{\infty}
t^{n} s_{n} (\,
\mathrm{Tr}( \Phi_{Z} ), 
\mathrm{Tr}( \Phi_{Z} (-\Psi_{Z}) ),
\ldots ,
\mathrm{Tr}( \Phi_{Z} (-\Psi_{Z})^{n-1} )
\,)
\end{split}
\end{equation*}

It remains to show that this right-hand side coincides with the right-hand side of the identity in Theorem~\ref{Thm:main}, that is,
$
\det ( 1 + t \Psi_{Z} )
\cdot \mathcal{E}_{Z} (t)
\cdot T_{Z} (t)
$.
To this end, we begin by computing
$
\mathcal{E}_{Z} (t)
\cdot T_{-Z} (t)
$.

\medskip

\begin{Lem} 
\label{Lem:QCMTV}
We have
$
\mathcal{E}_Z (t) \cdot T_{-Z}(t)
\;=\;
: \mathrm{e}^{ t D_{Z}^{*}} :
$
in $\mathcal{D}(R)[\![t]\!]$.
\end{Lem} 
\begin{proof} 
The coefficient of $\frac{t^n}{n!}$ in
$
\mathcal{E}_{Z} (t) \cdot T_{-Z}(t)
\;=\;
:\mathrm{e}^{t(D_{Z}+D_{Z}^{*})}:
\;
:\mathrm{e}^{ -tD_{Z} }:
$,
say $c_{n}$,
is given by
\begin{equation*}
\begin{split}
c_{n}
&=
\sum_{k=0}^{n}
\binom{n}{k}
\sum_{i_1, \cdots, i_n}
Z_{i_1} \cdots Z_{i_k}
: (D_{i_1}+D^*_{i_1}) \cdots (D_{i_k}+D^*_{i_k}) : \\
&\hspace{60mm}
\cdot
Z_{i_{k+1}} \cdots Z_{i_n} (-D_{i_{k+1}}) \cdots (-D_{i_n}).
\end{split}
\end{equation*}
For each summand, by the definition of normal order product, the following identity holds.
\begin{equation*}
\begin{split}
&
Z_{i_1} \cdots Z_{i_k}
{\color{light-gray}
\underbrace{
    {\color{black}
    : (D_{i_1}+D^*_{i_1}) \cdots (D_{i_k}+D^*_{i_k}) :
    }
}_{
    \begin{array}{c}
    \text{{\scriptsize This is a multiplication operator}}
    \end{array}
}
}
Z_{i_{k+1}} \cdots Z_{i_n} (-D_{i_{k+1}}) \cdots (-D_{i_n}) \\
&= 
Z_{i_1} \cdots Z_{i_k} Z_{i_{k+1}} \cdots Z_{i_n}
: (D_{i_1}+D^*_{i_1}) \cdots (D_{i_k}+D^*_{i_k}) :
(-D_{i_{k+1}}) \cdots (-D_{i_n}) \\
&= 
:
Z_{i_1} \cdots Z_{i_k} Z_{i_{k+1}} \cdots Z_{i_n}
(D_{i_1}+D^*_{i_1}) \cdots (D_{i_k}+D^*_{i_k})
(-D_{i_{k+1}}) \cdots (-D_{i_n})
:
\end{split}
\end{equation*}
Therefore we have the following.
\begin{equation*}
\begin{split}
c_{n}
&= 
\sum_{k=0}^{n}
\binom{n}{k}
:
\sum_{i_1, \cdots, i_k}
Z_{i_1} \cdots Z_{i_k}
(D_{i_1}+D^*_{i_1}) \cdots (D_{i_k}+D^*_{i_k}) \\
&\hspace{50mm}
\sum_{i_{k+1}, \cdots, i_{n}}
Z_{i_{k+1}} \cdots Z_{i_n}
(-D_{i_{k+1}}) \cdots (-D_{i_n})
: \\
&= 
\sum_{k=0}^{n}
\binom{n}{k}
:
\left( \sum_{i}
Z_{i}
(D_{i}+D^*_{i}) \right)^{k}
\left(
\sum_{i} (-Z_{i}) D_{i}
\right)^{n-k}
: \\
&= 
\sum_{k=0}^{n}
\binom{n}{k}
:
\left( D_{Z} + D_{Z}^{*} \right)^{k}
\left( -D_{Z} \right)^{n-k}
:
\;=\;
: \left( ( D_{Z} + D_{Z}^{*} ) - D_{Z} \right)^{n} :
\;=\;
: ( D_{Z}^{*} )^{n} :
\end{split}
\end{equation*}
The last expression is nothing but the coeficient of $\frac{t^{n}}{n!}$ in $:\mathrm{e}^{tD_{Z}^{*}}:$.
\end{proof} 

Next, we set
$
a(t) = \sum_{n=0}^{\infty} t^{n} a_{n} \in \mathcal{D}(R)[\![t]\!]
$
by the following formula.
\begin{equation}
\label{eq:a(t)_def}
a(t)
=
\det ( 1 + t \Psi_{Z} )
\cdot
\mathcal{E}_{Z} (t)
\cdot
T_{-Z} (t)
\end{equation}

For the next lemma, recall that $\Psi_{i,j} = [ D_{i}, Z_{j} ]$,
the $(i,j)$-th entry of the matrix $\Psi = \Psi_{Z}$.

\medskip

\begin{Lem} 
\label{Lem:an_decomp}
For $n \in \mathbb{N}$, we have
\begin{equation}
\label{eq:a_n}
\begin{split}
a_{n} 
&=
\frac{1}{n!}
\sum_{k=1}^{n}
\binom{n-1}{k-1}
\sum_{i_{1},\cdots ,i_{k}}
D_{i_{1}}^{*}
\mathrm{det}
\left( \begin{array}{ccc}
Z_{i_{1}} & \cdots & Z_{i_{k}} \\
\Psi_{i_{2},i_{1}} & \cdots & \Psi_{i_{2},i_{k}} \\
\vdots & \cdots & \vdots \\
\Psi_{i_{k},i_{1}} & \cdots & \Psi_{i_{k},i_{k}}
\end{array}\right)
:(D_{Z}^{*})^{n-k}:.
\end{split}
\end{equation}
\end{Lem} 

We would like to obtain a differential equation satisfied by $a(t)$ of the form $a^{\prime}(t) = b(t) a(t)$. In general, for
$c(t) = \sum_{n=0}^{\infty} t^{n} c_{n}$,
$e(t) = \sum_{n=0}^{\infty} t^{n} e_{n} \in \mathcal{D}(R)[\![ t ]\!]$,
the relation $a^{\prime}(t) = c(t) e(t)$ is equivalent, by comparing coefficients, to the following relations.
\begin{equation*}
a_{n} = \frac{1}{n} \sum_{k=1}^{n} c_{k-1} \, e_{n-k},
\quad
n \in \mathbb{N}.
\end{equation*}

We rewrite (\ref{eq:a_n}) as follows:
\begin{equation*}
\begin{split}
a_{n} 
&=
\frac{1}{n}
\sum_{k=1}^{n}
{\color{light-gray}
\underbrace{
    {\color{black}
    \frac{ 1 }{ (k-1)! }
    \sum_{i_{1},\cdots ,i_{k}}
    D_{i_{1}}^{*}
    \mathrm{det}
    \left( \begin{array}{ccc}
    Z_{i_{1}} & \cdots & Z_{i_{k}} \\
    \Psi_{i_{2},i_{1}} & \cdots & \Psi_{i_{2},i_{k}} \\
    \vdots & \cdots & \vdots \\
    \Psi_{i_{k},i_{1}} & \cdots & \Psi_{i_{k},i_{k}}
    \end{array}\right)
    }
}_{\text{
    the term corresponding to $c_{k-1}$
}}
}
{\color{light-gray}
\underbrace{
    {\color{black}
    \frac{ 1 }{ (n-k)! }
    :(D_{Z}^{*})^{n-k}:
    }
}_{\text{{\scriptsize $\displaystyle
    \begin{array}{c}
    \text{the term} \\
    \text{corresponding to $e_{n-k}$}
    \end{array}
$}}}
}.
\end{split}
\end{equation*}
we observe that the part corresponding to $e(t)$ appears in the form
$
e(t)
\;=\;
:\mathrm{e}^{tD_{Z}^{*}}:
\;=\;
\mathcal{E}_{Z}(t) \cdot T_{-Z}(t)
$.
By computing the term corresponding to $c(t)$, we see that $\det ( 1 + t \Psi_{Z} )$ appears as a factor on the right-hand side. Combining this with $e(t)$, we obtain
$
\det ( 1 + t \Psi_{Z} ) e(t)
=
\det ( 1 + t \Psi_{Z} ) \cdot \mathcal{E}_{Z}(t) \cdot T_{-Z}(t)
=
a(t)
$,
and hence we obtain a differential equation of the form $a^{\prime} (t) = b(t) a(t)$. That is, what we want to do next is to decompose $c(t)$, where $c(t)$ is given by
\begin{equation*}
c_{n}
=
\frac{ 1 }{ n! }
\sum_{i_{1},\cdots ,i_{n+1}}
D_{i_{1}}^{*}
\mathrm{det}
\left( \begin{array}{ccc}
Z_{i_{1}} & \cdots & Z_{i_{n+1}} \\
\Psi_{i_{2},i_{1}} & \cdots & \Psi_{i_{2},i_{n+1}} \\
\vdots & \cdots & \vdots \\
\Psi_{i_{n+1},i_{1}} & \cdots & \Psi_{i_{n+1},i_{n+1}}
\end{array} \right). 
\end{equation*}

\medskip

\begin{Lem}
\label{Lem:cn_decomp}
For each $n \in \mathbb{N}$, we have
\begin{equation*}
c_{n}
=
\sum_{k=0}^{n}
\mathrm{Tr}( \Phi_{Z} (-\Psi_{Z})^{k} )
\frac{ 1 }{ (n-k)! }
\sum_{ I \in \mathbb{N}^{n-k} }
\det ( \Psi_{Z} \mid I) .
\end{equation*}
\end{Lem} 

By this lemma, we see that the relation $c(t) = b(t) d(t)$ is satisfied by $b(t)$ and $d(t)$ defined as follows:
\begin{equation*}
\begin{split}
b(t)
&=
\sum_{n=0}^{\infty}
[\, \mathrm{Tr}( \Phi_{Z} (-\Psi_{Z})^{n} ) \,] \, t^{n}, \\
d(t)
&=
\sum_{n=0}^{\infty}
\frac{ t^{n} }{ n! }
\sum_{I \in \mathbb{N}^{n}} \det ( \Psi_{Z} \mid I )
=
\det ( 1 + t \Psi_{Z} ).
\end{split}
\end{equation*}
Now, since
$
a^{\prime} (t)
= c(t) e(t)
= b(t) d(t) e(t)
= b(t) a(t)
$
it follows that $a(t)$ satisfies the differential equation $a^{\prime}(t) = b(t) a(t)$. Since its solution is written as $\mathcal{T}\!\exp (\int_{0}^{t} b(t^{\prime}) \,\mathrm{d}t^{\prime})$, Theorem~\ref{Thm:main} follows.

It remains to prove Lemmas~\ref{Lem:an_decomp} and \ref{Lem:cn_decomp} introduced above.

\subsection{Proofs of Lemmas~\ref{Lem:an_decomp} and \ref{Lem:cn_decomp}}

\begin{proof}[Proof of Lemma~\ref{Lem:an_decomp}]
Let $n \in \mathbb{N}$ be arbitrary.
By (\ref{eq:det}) and
$
\mathcal{E}_{Z} \cdot T_{-Z}(t)
\;=\;
:\mathrm{e}^{ t D_{Z}^{*} }:
=
\sum_{n=1}^{\infty} \frac{t^{n}}{n!} :( D_{Z}^{*} )^{n}:
$
(Lemma~\ref{Lem:QCMTV}),
we have the following:
\begin{equation*}
a_{n} 
=
\frac{1}{n!}
\sum _{k=0}^{n} \binom{n}{k} 
\sum_{I \in \mathbb{N}^k } \det (\Psi_{Z} \mid I)
\;
:( D_{Z}^{*} )^{n-k}:.
\end{equation*}
In the case $n=1$, this becomes
$
:D_{Z}^{*}:
+
\sum_{i \in \mathbb{N}} \Psi_{i,i}
=
\sum_{i \in \mathbb{N}} ( Z_{i} D_{i}^{*} + [\, D_{i}^{*}, Z_{i} \,] )
=
\sum_{i \in \mathbb{N}} D_{i}^{*} Z_{i}
$,
which indeed coincides with the right-hand side of (\ref{eq:a_n}). Therefore, in what follows, we consider only the case $n\geq 2$.

Let $A=(A_{i,j})_{i,j\in\mathbb{N}}$ be an $\mathbb{N} \times \mathbb{N}$ matrix,
let $B=(B_{1}, B_{2}, B_{3}, \ldots)$ be a sequense,
and let
$I = ( i_{1}, i_{2}, \ldots , i_{n} ) \in \mathbb{N}^{n}$,
$J = ( j_{1}, j_{2}, \ldots , j_{m} ) \in \mathbb{N}^{m}$.
We define the submatrix $A_{J}^{I}$ and the subvector $B_{J}$ as follows:
\begin{equation*}
\begin{split}
A_{J}^{I}
=
\left(\begin{array}{cccc}
A_{i_{1},j_{1}} & A_{i_{1},j_{2}} & \cdots & A_{i_{1},j_{m}} \\
A_{i_{2},j_{1}} & A_{i_{2},j_{2}} & \cdots & A_{i_{2},j_{m}} \\
\vdots & \vdots & \cdots & \vdots \\
A_{i_{n},j_{1}} & A_{i_{n},j_{2}} & \cdots & A_{i_{n},j_{m}}
\end{array}\right) ,
\quad
B_{J}
=
( B_{j_{1}}, B_{j_{2}}, \ldots , B_{j_{n}} ).
\end{split}
\end{equation*}

We decompose the right-hand side of (\ref{eq:a_n}), which is the expression we wish to establish, as follows.
\begin{equation}
\label{eq:RHS_a_n}
\begin{split}
&
n! \cdot (\text{RHS of (\ref{eq:a_n})}) \\
&=
\sum_{k=1}^{n}
\binom{n-1}{k-1}
\sum_{i_{1},\cdots ,i_{k}}
D_{i_{1}}^{*}
\mathrm{det}
\left( \begin{array}{c}
Z_{ (i_{1}, i_{2}, \ldots , i_{k}) } \\\hdashline
\Psi_{ ( i_{1}, i_{2}, \ldots , i_{k} ) }^{ ( i_{2}, \ldots , i_{k} ) }
\end{array} \right)
:(D_{Z}^{*})^{n-k}: \\
&=
\sum_{k=1}^{n-1}
\binom{n-1}{k-1}
\sum_{ I \in \mathbb{N}^{n} }
D_{i_{1}}^{*}
\mathrm{det}
\left( \begin{array}{c}
Z_{ (i_{1}, i_{2}, \ldots , i_{k}) } \\\hdashline
\Psi_{ ( i_{1}, i_{2}, \ldots , i_{k} ) }^{ ( i_{2}, \ldots , i_{k} ) }
\end{array} \right)
Z_{i_{k+1}} Z_{i_{k+2}} \cdots Z_{i_{n}}
D_{i_{k+1}}^{*} D_{i_{k+2}}^{*} \cdots D_{i_{n}}^{*} \\
&\hspace{25mm}+
\sum_{ I \in \mathbb{N}^{n} }
D_{i_{1}}^{*}
\mathrm{det}
\left( \begin{array}{c}
Z_{ ( i_{1}, i_{2}, \ldots , i_{n} ) } \\\hdashline
\Psi_{ ( i_{1}, i_{2}, \ldots , i_{n} ) }^{ ( i_{2}, \ldots , i_{n} ) }
\end{array} \right) .
\end{split}
\end{equation}

Let
$I = (i_{1}, i_{2}, \ldots , i_{n}) \in \mathbb{N}^{n}$
and
$k=1,2,\ldots , n-1$ be arbitrary.
By the definition of commutator, we have
\begin{equation*}
\begin{split}
&
D_{i_{1}}^{*}
\mathrm{det}
\left(\begin{array}{c}
Z_{ (i_{1}, i_{2}, \ldots , i_{k}) } \\\hdashline
\Psi_{ (i_{1}, i_{2}, \ldots , i_{k}) }^{ (i_{2}, \ldots , i_{k}) }
\end{array}\right)
Z_{i_{k+1}} \cdots Z_{i_{n}} \\
&=
[\,
D_{i_{1}}^{*},
\mathrm{det}
\left(\begin{array}{c}
Z_{ ( i_{1}, i_{2}, \ldots , i_{k} ) } \\\hdashline
\Psi_{ (i_{1}, i_{2}, \ldots , i_{k}) }^{ (i_{2}, \ldots , i_{k}) }
\end{array}\right)
Z_{i_{k+1}} \cdots Z_{i_{n}}
\,]
+
\mathrm{det}
\left(\begin{array}{c}
Z_{ ( i_{1}, i_{2}, \ldots , i_{k} ) } \\\hdashline
\Psi_{ (i_{1}, i_{2}, \ldots , i_{k}) }^{ (i_{2}, \ldots , i_{k}) }
\end{array}\right)
Z_{i_{k+1}} \cdots Z_{i_{n}}
D_{i_{1}}^{*}.
\end{split}
\end{equation*}
For the first term of the right-hand side, we apply the Leibniz' rule of the Lie bracket, which gives the following:
\begin{equation*}
\begin{split}
&
[\,
D_{i_{1}}^{*},
\mathrm{det}
\left(\begin{array}{c}
Z_{ ( i_{1}, i_{2}, \ldots , i_{k} ) } \\\hdashline
\Psi_{ (i_{1}, i_{2}, \ldots , i_{k}) }^{ (i_{2}, \ldots , i_{k}) }
\end{array}\right)
Z_{i_{k+1}} \cdots Z_{i_{n}}
\,] \\
&= 
[\,
D_{i_{1}}^{*},
\mathrm{det}
\left(\begin{array}{c}
Z_{ ( i_{1}, i_{2}, \ldots , i_{k} ) } \\\hdashline
\Psi_{ (i_{1}, i_{2}, \ldots , i_{k}) }^{ (i_{2}, \ldots , i_{k}) }
\end{array}\right)
\,]\,
Z_{i_{k+1}} \cdots Z_{i_{n}} \\
&\hspace{30mm}+
\sum_{j=1}^{n-k}
\mathrm{det}
\left(\begin{array}{c}
Z_{ ( i_{1}, i_{2}, \ldots , i_{k} ) } \\\hdashline
\Psi_{ (i_{1}, i_{2}, \ldots , i_{k}) }^{ (i_{2}, \ldots , i_{k}) }
\end{array}\right)
Z_{i_{k+1}}
\cdots
[\, D_{i_{1}}^{*}, Z_{i_{k+j}} \,]
\cdots Z_{i_{n}} .
\end{split}
\end{equation*}
Substituting this into the expression above and then multiplying by $D_{i_{k+1}}^{*} \cdots D_{i_{n}}^{*}$ from the right, we obtain the following decomposition.
\begin{equation*}
\begin{split}
&
D_{i_{1}}^{*}
\mathrm{det}
\left(\begin{array}{c}
Z_{ ( i_{1}, i_{2}, \ldots , i_{k} ) }  \\\hdashline
\Psi_{ (i_{1}, i_{2}, \ldots , i_{k}) }^{ (i_{2}, \ldots , i_{k}) }
\end{array}\right)
Z_{i_{k+1}} \cdots Z_{i_{n}}
D_{i_{k+1}}^{*} \cdots D_{i_{n}}^{*}\\
&= 
\underbrace{
    [\,
    D_{i_{1}}^{*},
    \mathrm{det}
    \left(\begin{array}{c}
    Z_{ ( i_{1}, i_{2}, \ldots , i_{k} ) } \\\hdashline
    \Psi_{ (i_{1}, i_{2}, \ldots , i_{k}) }^{ (i_{2}, \ldots , i_{k}) }
    \end{array}\right)
    \,]\,
    Z_{i_{k+1}} \cdots Z_{i_{n}}
    D_{i_{k+1}}^{*} \cdots D_{i_{n}}^{*}
}_{\text{{\small $\displaystyle
    \phantom{A_{I,k}} = A_{I,k}
$}}} \\
&\hspace{10mm}+
\underbrace{
    \sum_{l=1}^{n-k}
    \mathrm{det}
    \left(\begin{array}{c}
    Z_{ ( i_{1}, i_{2}, \ldots , i_{k} ) } \\\hdashline
    \Psi_{ (i_{1}, i_{2}, \ldots , i_{k}) }^{ (i_{2}, \ldots , i_{k}) }
    \end{array}\right)
    Z_{i_{k+1}}
    \cdots
    {\color{light-gray}
    \underbrace{
        {\color{black}
        [\, D_{i_{1}}^{*}, Z_{i_{k+l}} \,]
        }
    }_{
        \phantom{-\Psi_{i_{1}, i_{k+l}}}
        =
        - \Psi_{i_{1}, i_{k+l}}
    }
    }
    \cdots Z_{i_{n}}
    D_{i_{k+1}}^{*} \cdots D_{i_{n}}^{*}
}_{\text{{\small $\displaystyle
    \phantom{B_{I,k}} = B_{I,k}
$}}} \\
&\hspace{20mm}+
\underbrace{
    \mathrm{det}
    \left(\begin{array}{c}
    Z_{ ( i_{1}, i_{2}, \ldots , i_{k} ) } \\\hdashline
    \Psi_{ (i_{1}, i_{2}, \ldots , i_{k}) }^{ (i_{2}, \ldots , i_{k}) }
    \end{array}\right)
    Z_{i_{k+1}} \cdots Z_{i_{n}}
    D_{i_{1}}^{*} D_{i_{k+1}}^{*} \cdots D_{i_{n}}^{*}
}_{\text{{\small $\displaystyle
    \phantom{C_{I,k}} = C_{I,k}
$}}} .
\end{split}
\end{equation*}
Here, we define $A_{I,k}$, $B_{I,k}$ and $C_{I,k}$ as above.
By convention, we set $B_{I,n} = 0$.
It then follows from (\ref{eq:RHS_a_n}) that
\begin{equation}
\label{eq:RHS}
\begin{split}
&
n! \cdot (\text{RHS of (\ref{eq:a_n})}) \\
&= 
\sum_{k=1}^{n-1}
\binom{n-1}{k-1}
\sum_{ I \in \mathbb{N}^{n} }
{\color{light-gray}
\underbrace{
    {\color{black}
    D_{i_{1}}^{*}
    \mathrm{det}
    \left( \begin{array}{c}
    Z_{ (i_{1}, i_{2}, \ldots , i_{k}) } \\\hdashline
    \Psi_{ ( i_{1}, i_{2}, \ldots , i_{k} ) }^{ ( i_{2}, \ldots , i_{k} ) }
    \end{array} \right)
    Z_{i_{k+1}} Z_{i_{k+2}} \cdots Z_{i_{n}}
    D_{i_{k+1}}^{*} D_{i_{k+2}}^{*} \cdots D_{i_{n}}^{*}
    }
}_{\text{{\small $\displaystyle
    \phantom{A_{I,k} + B_{I,k} + C_{I,k}}
    =
    A_{I,k} + B_{I,k} + C_{I,k}
$}}}
} \\
&\hspace{25mm}+
\sum_{ I \in \mathbb{N}^{n} }
{\color{light-gray}
\underbrace{
    {\color{black}
    D_{i_{1}}^{*}
    \mathrm{det}
    \left(\begin{array}{c}
    Z_{ (i_{1}, i_{2}, \ldots , i_{n}) } \\\hdashline
    \Psi_{ ( i_{1}, i_{2}, \ldots , i_{n} ) }^{ ( i_{2}, \ldots , i_{n} ) }
    \end{array}\right)
    }
}_{\text{{\small $\displaystyle
    \phantom{C_{I,n}}
    =
    C_{I,n}
$}}}
} \\
&= 
\sum_{ I \in \mathbb{N}^{n} }
\sum_{k=1}^{n-1}
\binom{n-1}{k-1}
\left( A_{I,k} + B_{I,k} + C_{I,k} \right)
+
\sum_{ I \in \mathbb{N}^{n} }
(A_{I,n} + C_{I,n}) \\
&= 
\sum_{ I \in \mathbb{N}^{n} }
\Bigg(
    C_{I,1}
    +
    \sum_{k=1}^{n-1}
    \left\{
        \binom{n-1}{k-1}
        \left( A_{I,k} + B_{I,k} \right)
        +
        \binom{n-1}{k} C_{I,k+1}
    \right\}
    + A_{I,n}
\Bigg) \\
&= 
\hspace{-22mm}
{\color{light-gray}
    \hspace{-22mm}
    {\color{light-gray}
    \underbrace{
        {\color{black}
        \sum_{ I \in \mathbb{N}^{n} } C_{I,1}
        }
    }_{\text{{\scriptsize $\displaystyle
        \begin{array}{c}
        \phantom{\displaystyle \sum_{ I \in \mathbb{N}^{n}} Z_{i_{1}} Z_{i_{2}} \cdots Z_{i_{n}}
        D_{i_{1}}^{*} D_{i_{2}}^{*} \cdots D_{i_{n}}^{*}}
        = \displaystyle \sum_{ I \in \mathbb{N}^{n}} Z_{i_{1}} Z_{i_{2}} \cdots Z_{i_{n}}
        D_{i_{1}}^{*} D_{i_{2}}^{*} \cdots D_{i_{n}}^{*} \\
        \phantom{\; :(D_{Z}^{*})^{n}:}
        =\; :(D_{Z}^{*})^{n}:
        \end{array}
    $}}}
    }
    \hspace{-26mm}
}
\hspace{-11mm}
+
\sum_{k=1}^{n-1}
\Bigg\{
    {\color{light-gray}
    \underbrace{
        {\color{black}
        \binom{n-1}{k-1}
        \sum_{ I \in \mathbb{N}^{n} }
        \left( A_{I,k} + B_{I,k} \right)
        +
        \binom{n-1}{k}
        \sum_{ I \in \mathbb{N}^{n} }
        C_{I,k+1}
        }
    }_{\text{{\scriptsize
        will be computed in what follows
    }}}
    }
\Bigg\}
+\hspace{-22mm}
    {\color{light-gray}
    \underbrace{
        {\color{black}
        \sum_{I \in \mathbb{N}^{n}} A_{I,n}
        }
    }_{\text{{\scriptsize $\displaystyle
    \phantom{\sum_{I \in \mathbb{N}^{n}} \det ( \Psi_{Z} \mid I )} \stackrel{\text{Lemma~\ref{Lem:adin}}}{=}
    \sum_{I \in \mathbb{N}^{n}} \det ( \Psi_{Z} \mid I )
    $}}}
    }
    \hspace{-11.5mm}.
\end{split}
\end{equation}

In what follows, we compute the second term on the right-hand side of the above expression. Let $k$ be arbitrary such that $1 \leq k \leq n-1$. First, in order to compute $\sum_{I \in \mathbb{N}^{n}} B_{I,k}$, we fix $l$ arbitrariry such that $1 \leq l \leq n-k$, and define a transformation $I \mapsto J = (j_{1}, j_{2}, \ldots , j_{n})$ of $\mathbb{N}^{n}$ by the following relation.
$$
I = (
\underbrace{i_{1}, \ldots , i_{k}}_{
    \begin{array}{c}
    \text{{\small \rotatebox{90}{$=$}}} \vspace{-2mm}\\
    \text{{\small $j_{1}, \ldots , j_{k}$}}
    \end{array}
},
\underbrace{i_{k+1}, \ldots , i_{k+l-1}}_{
    \begin{array}{c}
    \text{{\small \rotatebox{90}{$=$}}} \vspace{-2mm}\\
    \text{{\small $j_{k+2}, \ldots , j_{k+l}$}}
    \end{array}
},
\underbrace{i_{k+l}}_{
    \begin{array}{c}
    \text{{\small \rotatebox{90}{$=$}}} \vspace{-2mm}\\
    \text{{\small $j_{k+1}$}}
    \end{array}
},
\underbrace{i_{k+l+1}, \ldots , i_{n}}_{
    \begin{array}{c}
    \text{{\small \rotatebox{90}{$=$}}} \vspace{-2mm}\\
    \text{{\small $j_{k+l+1}, \ldots , j_{n}$}}
    \end{array}
}
)
$$
Then the computation of $\sum_{I \in \mathbb{N}^{n}} B_{I,k}$ goes as follows:
\begin{equation}
\label{eq:B_sum}
\begin{split}
\sum_{I \in \mathbb{N}^{n}} B_{I,k}
&= 
\sum_{l=1}^{n-k}
\sum_{I \in \mathbb{N}^{n}}
{\color{light-gray}
\underbrace{
    {\color{black}
    \mathrm{det}
    \left(\begin{array}{c}
    Z_{ ( i_{1}, i_{2}, \ldots , i_{k} ) } \\\hdashline
    \Psi_{ (i_{1}, i_{2}, \ldots , i_{k}) }^{ (i_{2}, \ldots , i_{k}) }
    \end{array}\right)
    Z_{i_{k+1}}
    \cdots
    Z_{i_{k+l-1}}
    \Psi_{i_{1}, i_{k+l}}
    Z_{i_{k+l+1}}
    \cdots Z_{i_{n}}
    D_{i_{k+1}}^{*} \cdots D_{i_{n}}^{*}
    }
}_{\text{{\scriptsize $\displaystyle
    \phantom{ \Psi_{j_{1}, j_{k+1}} \mathrm{det} }
    =
    \Psi_{j_{1}, j_{k+1}}
    \mathrm{det}
    \left(\begin{array}{c}
    Z_{ ( j_{1}, j_{2},  \ldots , j_{k} ) } \\\hdashline
    \Psi_{ (j_{1}, j_{2}, \ldots , j_{k}) }^{ (j_{2}, \ldots , j_{k}) }
    \end{array}\right)
    Z_{j_{k+2}} \cdots Z_{j_{n}}
    D_{j_{k+2}}^{*} \cdots D_{j_{n}}^{*}
$}}}
}\\
&=
\sum_{l=1}^{n-k}
{\color{light-gray}
\underbrace{
    {\color{black}
    \sum_{I \in \mathbb{N}^{n}}
    \Psi_{i_{1}, i_{k+1}}
    \mathrm{det}
    \left(\begin{array}{c}
    Z_{ ( i_{1}, i_{2}, \ldots , i_{k} ) } \\\hdashline
    \Psi_{ (i_{1}, i_{2}, \ldots , i_{k}) }^{ (i_{2}, \ldots , i_{k}) }
    \end{array}\right)
    Z_{i_{k+2}} \cdots Z_{i_{n}}
    D_{i_{k+2}}^{*} \cdots D_{i_{n}}^{*}
    }
}_{\text{doesn't depend on $l$}}
} \\
&= 
(n-k)
\sum_{I \in \mathbb{N}^{n}}
\Psi_{i_{1}, i_{k+1}}
\mathrm{det}
\left(\begin{array}{c}
Z_{ ( i_{1}, i_{2}, \ldots , i_{k} ) } \\\hdashline
\Psi_{ (i_{1}, i_{2}, \ldots , i_{k}) }^{ (i_{2}, \ldots , i_{k}) }
\end{array}\right)
Z_{i_{k+2}} \cdots Z_{i_{n}}
D_{i_{k+2}}^{*} \cdots D_{i_{n}}^{*}
\end{split}
\end{equation}

Next, we focus on $C_{I, k+1}$.
By the cofactor expansion along the first column of the determinant in $C_{I, k+1}$, for $k = 1, \cdots, n-1$, we obtain the following.
\begin{equation*}
\begin{split}
C_{I,k+1}
&= 
\mathrm{det}
\left(\begin{array}{c:ccc}
Z_{i_{1}} & Z_{i_{2}} & \cdots & Z_{i_{k+1}} \\\hdashline
\Psi_{i_{2}, i_{1}} & \Psi_{i_{2}, i_{2}} & \cdots & \Psi_{i_{2}, i_{k+1}} \\
\vdots & \vdots & \cdots & \vdots \\
\Psi_{i_{k+1}, i_{1}} & \Psi_{i_{k+1}, i_{2}} & \cdots & \Psi_{i_{k+1}, i_{k+1}}
\end{array}\right)
Z_{i_{k+2}} \cdots Z_{i_{n}}
D_{i_{1}}^{*} D_{i_{k+2}}^{*} \cdots D_{i_{n}}^{*} \\
&= 
Z_{i_{1}}
\mathrm{det}
(\Psi_{ Z } \mid (i_{2},\ldots , i_{k+1})
)
\cdot
Z_{i_{k+2}} \cdots Z_{i_{n}}
D_{i_{1}}^{*} D_{i_{k+2}}^{*} \cdots D_{i_{n}}^{*} \\
&\hspace{10mm}+
\sum_{l=2}^{k+1}
(-1)^{l-1}
\Psi_{i_{l}, i_{1}}
\det \left(\begin{array}{ccc}
Z_{i_{2}} & \cdots & Z_{i_{k+1}} \\\hdashline
\Psi_{i_{2}, i_{2}} & \cdots & \Psi_{i_{2}, i_{k+1}} \\
\vdots & \cdots & \vdots \\
\Psi_{i_{l-1}, i_{2}} & \cdots & \Psi_{i_{l-1}, i_{k+1}} \\
\Psi_{i_{l+1}, i_{2}} & \cdots & \Psi_{i_{l+1}, i_{k+1}} \\
\vdots & \cdots & \vdots \\
\Psi_{i_{k+1}, i_{2}} & \cdots & \Psi_{i_{k+1}, i_{k+1}}
\end{array}\right)
\cdot
Z_{i_{k+2}} \cdots Z_{i_{n}}
D_{i_{1}}^{*} D_{i_{k+2}}^{*} \cdots D_{i_{n}}^{*}
\end{split}
\end{equation*}
We now observe the following.
\begin{equation*}
\begin{split}
&
\det \left(\begin{array}{ccc}
Z_{i_{2}} & \cdots & Z_{i_{k+1}} \\\hdashline
\Psi_{i_{2}, i_{2}} & \cdots & \Psi_{i_{2}, i_{k+1}} \\
\vdots & \cdots & \vdots \\
\Psi_{i_{l-1}, i_{2}} & \cdots & \Psi_{i_{l-1}, i_{k+1}} \\\hdashline
\Psi_{i_{l+1}, i_{2}} & \cdots & \Psi_{i_{l+1}, i_{k+1}} \\
\vdots & \cdots & \vdots \\
\Psi_{i_{k+1}, i_{2}} & \cdots & \Psi_{i_{k+1}, i_{k+1}}
\end{array}\right)
=
\det \left(\begin{array}{cc:c:cc}
Z_{i_{2}} & \cdots & Z_{i_{l}} & \cdots & Z_{i_{k+1}} \\\hdashline
\Psi_{i_{2}, i_{2}} & \cdots & \Psi_{i_{2}, i_{l}} & \cdots & \Psi_{i_{2}, i_{k+1}} \\
\vdots & \cdots & \vdots & \cdots & \vdots \\
\Psi_{i_{l-1}, i_{2}} & \cdots & \Psi_{i_{l-1}, i_{l}} & \cdots & \Psi_{i_{l-1}, i_{k+1}} \\\hdashline
\Psi_{i_{l+1}, i_{2}} & \cdots & \Psi_{i_{l+1}, i_{l}} & \cdots & \Psi_{i_{l+1}, i_{k+1}} \\
\vdots & \cdots & \vdots & \cdots & \vdots \\
\Psi_{i_{k+1}, i_{2}} & \cdots & \Psi_{i_{k+1}, i_{l}} & \cdots & \Psi_{i_{k+1}, i_{k+1}}
\end{array}\right) \\
&=
(-1)^{l-2}
\det \left(\begin{array}{c:ccc:ccc}
Z_{i_{l}} & Z_{i_{2}} & \cdots & Z_{i_{l-1}} & Z_{i_{l+1}} & \cdots & Z_{i_{k+1}} \\\hdashline
\Psi_{i_{2}, i_{l}} & \Psi_{i_{2}, i_{2}} & \cdots & \Psi_{i_{2}, i_{l-1}} & \Psi_{i_{2}, i_{l+1}} & \cdots & \Psi_{i_{2}, i_{k+1}} \\
\vdots & \vdots & \cdots & \vdots & \vdots & \cdots & \vdots \\
\Psi_{i_{l-1}, i_{l}} & \Psi_{i_{l-1}, i_{2}} & \cdots & \Psi_{i_{l-1}, i_{l-1}} & \Psi_{i_{l-1}, i_{l+1}} & \cdots & \Psi_{i_{l-1}, i_{k+1}} \\\hdashline
\Psi_{i_{l+1}, i_{l}} & \Psi_{i_{l+1}, i_{2}} & \cdots & \Psi_{i_{l+1}, i_{l-1}} & \Psi_{i_{l+1}, i_{l+1}} & \cdots & \Psi_{i_{l+1}, i_{k+1}} \\
\vdots & \vdots & \cdots &  & \vdots & \cdots & \vdots \\
\Psi_{i_{k+1}, i_{l}} & \Psi_{i_{k+1}, i_{2}} & \cdots & \Psi_{i_{k+1}, i_{l-1}} & \Psi_{i_{k+1}, i_{l-+}} & \cdots & \Psi_{i_{k+1}, i_{k+1}}
\end{array}\right)
\end{split}
\end{equation*}
Consequently, $C_{I,k+1}$ admits the following representation.
\begin{equation*}
\begin{split}
C_{I,k+1}
&= 
\mathrm{det}
(\Psi_{ Z } \mid (i_{2},\ldots , i_{k+1})
)
\cdot
Z_{i_{1}} Z_{i_{k+2}} \cdots Z_{i_{n}}
D_{i_{1}}^{*} D_{i_{k+2}}^{*} \cdots D_{i_{n}}^{*} \\
&\hspace{10mm}-
\sum_{l=2}^{k+1}
\Psi_{i_{l}, i_{1}}
\det \left(\begin{array}{c:c}
Z_{i_{l}} & Z_{ ( \widehat{i_{1}}, i_{2}, \ldots , \widehat{i_{l}}, \ldots , i_{k+1}) } \\\hdashline
\Psi_{i_{l}}^{ ( \widehat{i_{1}}, i_{2}, \ldots , \widehat{i_{l}}, \ldots , i_{k+1}) }
& \Psi_{ ( \widehat{i_{1}}, i_{2}, \ldots , \widehat{i_{l}}, \ldots , i_{k+1}) }^{ ( \widehat{i_{1}}, i_{2}, \ldots , \widehat{i_{l}}, \ldots , i_{k+1}) }
\end{array}\right)
\cdot
Z_{i_{k+2}} \cdots Z_{i_{n}}
D_{i_{1}}^{*} D_{i_{k+2}}^{*} \cdots D_{i_{n}}^{*}
\end{split}
\end{equation*}
Summing both sides over
$I = (i_{1}, i_{2}, \ldots , i_{k+1}) \in \mathbb{N}^{n}$,
we obtain the following expression for $\sum_{ I \in \mathbb{N}^{n} } C_{I,k+1}$.
\begin{equation}
\label{eq:C_sum}
\begin{split}
&
\sum_{ I \in \mathbb{N}^{n} } C_{I,k+1} \\
&= 
\sum_{ I \in \mathbb{N}^{n} }
\mathrm{det}
(\Psi_{ Z } \mid (i_{2},\ldots , i_{k+1})
)
\cdot
Z_{i_{1}} Z_{i_{k+2}} \cdots Z_{i_{n}}
D_{i_{1}}^{*} D_{i_{k+2}}^{*} \cdots D_{i_{n}}^{*} \\
&\hspace{10mm}-
\sum_{l=2}^{k+1}
\sum_{ I \in \mathbb{N}^{n} }
\Psi_{i_{l}, i_{1}}
\det \left(\begin{array}{c:c}
Z_{i_{l}} & Z_{ ( \widehat{i_{1}}, i_{2}, \ldots , \widehat{i_{l}}, \ldots , i_{k+1}) } \\\hdashline
\Psi_{i_{l}}^{ ( \widehat{i_{1}}, i_{2}, \ldots , \widehat{i_{l}}, \ldots , i_{k+1}) }
& \Psi_{ ( \widehat{i_{1}}, i_{2}, \ldots , \widehat{i_{l}}, \ldots , i_{k+1}) }^{ ( \widehat{i_{1}}, i_{2}, \ldots , \widehat{i_{l}}, \ldots , i_{k+1}) }
\end{array}\right)
\cdot
Z_{i_{k+2}} \cdots Z_{i_{n}}
D_{i_{1}}^{*} D_{i_{k+2}}^{*} \cdots D_{i_{n}}^{*}
\end{split}
\end{equation}
Let us consider the first term on the right-hand side. Define a transformation $I \mapsto J = (j_{1}, j_{2}, \ldots , j_{n})$ of $\mathbb{N}^{n}$ by the following relation.
$$
I = (
\underbrace{i_{1}}_{
    \begin{array}{c}
    \text{{\small \rotatebox{90}{$=$}}} \vspace{-2mm}\\
    \text{{\small $j_{k+1}$}}
    \end{array}
},
\underbrace{i_{2}, \ldots , i_{k+1}}_{
    \begin{array}{c}
    \text{{\small \rotatebox{90}{$=$}}} \vspace{-2mm}\\
    \text{{\small $j_{1}, \ldots , j_{k}$}}
    \end{array}
},
\underbrace{i_{k+2}, \ldots , i_{n}}_{
    \begin{array}{c}
    \text{{\small \rotatebox{90}{$=$}}} \vspace{-2mm}\\
    \text{{\small $j_{k+2}, \ldots , j_{n}$}}
    \end{array}
}
)
$$
Then, by performing this change of variables, we obtain the following expression.
\begin{equation*}
\begin{split}
&
\sum_{ I \in \mathbb{N}^{n} }
{\color{light-gray}
    {\color{black}
    \mathrm{det} \left(
    \Psi_{Z} \mid (i_{2},\ldots , i_{k+1})\right)
    \cdot
    Z_{i_{1}} Z_{i_{k+2}} \cdots Z_{i_{n}}
    D_{i_{1}}^{*}
    D_{i_{k+2}}^{*} \cdots D_{i_{n}}^{*}
    }
} \\
&=
\sum_{I \in \mathbb{N}^{n}}
\det ( \Psi_{Z} \mid (i_{1}, \ldots , i_{k}) )
\cdot
Z_{i_{k+1}} Z_{i_{k+2}} \cdots Z_{i_{n}}
D_{i_{k+1}}^{*}
D_{i_{k+2}}^{*} \cdots D_{i_{n}}^{*}
\end{split}
\end{equation*}
Next, we focus on the term $\sum_{ I \in \mathbb{N}^{n} } (...)$ appearing as the second term on the right-hand side of (\ref{eq:C_sum}). In this sum, an index $l$ with $2 \leq l \leq k+1$ is fixed. Using this $l$, we define a transformation $I \mapsto J = (j_{1}, j_{2}, \ldots , j_{k+1})$ of $\mathbb{N}^{n}$ by thefollowing relation.
$$
I = (
\underbrace{i_{1}}_{
    \begin{array}{c}
    \text{{\small \rotatebox{90}{$=$}}} \vspace{-2mm}\\
    \text{{\small $j_{k+1}$}}
    \end{array}
},
\underbrace{i_{2}, \ldots , i_{l-1}}_{
    \begin{array}{c}
    \text{{\small \rotatebox{90}{$=$}}} \vspace{-2mm}\\
    \text{{\small $j_{2}, \ldots , j_{l-1}$}}
    \end{array}
},
\underbrace{i_{l}}_{
    \begin{array}{c}
    \text{{\small \rotatebox{90}{$=$}}} \vspace{-2mm}\\
    \text{{\small $j_{1}$}}
    \end{array}
},
\underbrace{i_{l+1}, \ldots , i_{k+1}}_{
    \begin{array}{c}
    \text{{\small \rotatebox{90}{$=$}}} \vspace{-2mm}\\
    \text{{\small $j_{l}, \ldots , j_{k}$}}
    \end{array}
},
\underbrace{i_{k+2}, \ldots , i_{n}}_{
    \begin{array}{c}
    \text{{\small \rotatebox{90}{$=$}}} \vspace{-2mm}\\
    \text{{\small $j_{k+1}, \ldots , j_{n}$}}
    \end{array}
}
)
$$
Then, the corresponding part of the summand in $\sum_{ I \in \mathbb{N}^{n} } (...)$ can be expressed as follows.
\begin{equation*}
\begin{split}
&
\Psi_{i_{l}, i_{1}}
\det \left(\begin{array}{c:c}
Z_{i_{l}} & Z_{ ( \widehat{i_{1}}, i_{2}, \ldots , \widehat{i_{l}}, \ldots , i_{k+1}) } \\\hdashline
\Psi_{i_{l}}^{ ( \widehat{i_{1}}, i_{2}, \ldots , \widehat{i_{l}}, \ldots , i_{k+1}) }
& \Psi_{ ( \widehat{i_{1}}, i_{2}, \ldots , \widehat{i_{l}}, \ldots , i_{k+1}) }^{ ( \widehat{i_{1}}, i_{2}, \ldots , \widehat{i_{l}}, \ldots , i_{k+1}) }
\end{array}\right)
\cdot
Z_{i_{k+2}} \cdots Z_{i_{n}}
D_{i_{1}}^{*} D_{i_{k+2}}^{*} \cdots D_{i_{n}}^{*} \\
&=
\Psi_{j_{1}, j_{k+1}}
{\color{light-gray}
\underbrace{
    {\color{black}
    \det \left(\begin{array}{c:c}
    Z_{j_{1}} & Z_{ ( j_{2}, \ldots , j_{k}) } \\\hdashline
    \Psi_{j_{1}}^{ ( j_{2}, \ldots , j_{k}) }
    & \Psi_{ ( j_{2}, \ldots , j_{k}) }^{ ( j_{2}, \ldots , j_{k}) }
    \end{array}\right)
    }
}_{\text{{\scriptsize $\displaystyle
    \phantom{\det
    \left(\begin{array}{c}
    Z_{(j_{1}, j_{2}, \ldots , j_{k})} \\\hdashline
    \Psi_{(j_{1}, j_{2} \ldots , j_{k})}^{(j_{2}, \ldots , j_{k})}
    \end{array}\right)}
    =
    \det
    \left(\begin{array}{c}
    Z_{(j_{1}, j_{2}, \ldots , j_{k})} \\\hdashline
    \Psi_{(j_{1}, j_{2} \ldots , j_{k})}^{(j_{2}, \ldots , j_{k})}
    \end{array}\right)
$}}}
}
\cdot
Z_{j_{k+2}} \cdots Z_{j_{n}}
D_{j_{k+1}}^{*} D_{j_{k+2}}^{*} \cdots D_{j_{n}}^{*}
\end{split}
\end{equation*}
Using this expression, we obtain the following computation of $\sum_{ I \in \mathbb{N}^{n} } C_{I,k+1}$.
\begin{equation*}
\begin{split}
&
\sum_{ I \in \mathbb{N}^{n} } C_{I,k+1} \\
&= 
\sum_{I \in \mathbb{N}^{n}}
\det ( \Psi_{Z} \mid (i_{1}, \ldots , i_{k}) )
\cdot
Z_{i_{k+1}} Z_{i_{k+2}} \cdots Z_{i_{n}}
D_{i_{k+1}}^{*}
D_{i_{k+2}}^{*} \cdots D_{i_{n}}^{*} \\
&\hspace{10mm}-
\sum_{l=2}^{k+1}
{\color{light-gray}
\underbrace{
    {\color{black}
    \sum_{ J \in \mathbb{N}^{n} }
    \Psi_{j_{1}, j_{k+1}}
    \det
    \left(\begin{array}{c}
    Z_{(j_{1}, j_{2}, \ldots , j_{k})} \\\hdashline
    \Psi_{(j_{1}, j_{2} \ldots , j_{k})}^{(j_{2}, \ldots , j_{k})}
    \end{array}\right)
    \cdot
    Z_{j_{k+2}} \cdots Z_{j_{n}}
    D_{j_{k+1}}^{*} D_{j_{k+2}}^{*} \cdots D_{j_{n}}^{*}
    }
}_{
    \text{{\scriptsize doesn't depend on $l$}}
}
} \\
&= 
\sum_{I \in \mathbb{N}^{n}}
\det ( \Psi_{Z} \mid (i_{1}, \ldots , i_{k}) )
\cdot
Z_{i_{k+1}} Z_{i_{k+2}} \cdots Z_{i_{n}}
D_{i_{k+1}}^{*}
D_{i_{k+2}}^{*} \cdots D_{i_{n}}^{*}
-
\frac{k}{n-k}
    \sum_{I \in \mathbb{N}^{n}} B_{I,k}
\end{split}
\end{equation*}
We now consider the summand of $\sum_{k=1}^{n-1}(...)$ appearing as the second term on the far right-hand side of (\ref{eq:RHS}). Then the following holds.
\begin{equation*}
\begin{split}
&
\binom{n-1}{k} \sum_{I \in \mathbb{N}^{n}} C_{I,k+1}
+
\binom{n-1}{k-1} \sum_{I \in \mathbb{N}^{n}} B_{I,k}
+
\binom{n-1}{k-1} \sum_{I \in \mathbb{N}^{n}} A_{I,k}\\
&= 
\binom{n-1}{k}
\sum_{I \in \mathbb{N}^{n}}
\det ( \Psi \mid (i_{1}, \ldots , i_{k}) )
\cdot
Z_{i_{k+1}} Z_{i_{k+2}} \cdots Z_{i_{n}}
D_{i_{k+1}}^{*}
D_{i_{k+2}}^{*} \cdots D_{i_{n}}^{*} \\
&\hspace{10mm}
{\color{light-gray}
\underbrace{
    {\color{black}
    -
    \hspace{-5mm}
    {\color{light-gray}
    \underbrace{
        {\color{black}
        k \binom{n-1}{k}
        }
    }_{\text{{\scriptsize $
        \phantom{(n-k) \binom{n-1}{k-1}}
        =
        (n-k) \binom{n-1}{k-1}
    $}}}
    }
    \hspace{-8mm}
    \frac{1}{n-k}
    \sum_{I \in \mathbb{N}^{n}} B_{I,k}
    +
    \binom{n-1}{k-1}
    \sum_{I \in \mathbb{N}^{n}} B_{I,k}
    }
}_{\text{{\scriptsize $\displaystyle
    \phantom{0} = 0
$}}}
}\\
&\hspace{15mm}+
\binom{n-1}{k-1}
\sum_{i_{k+1}, \ldots , i_{n}}
{\color{light-gray}
\underbrace{
    {\color{black}
    \sum_{i_{1}, \ldots , i_{k}}
    [\,
    D_{i_{1}}^{*},
    \mathrm{det}
    \left(\begin{array}{c}
    Z_{ ( i_{1}, \ldots , i_{k} ) } \\\hdashline
    \Psi_{ (i_{1}, i_{2}, \ldots , i_{k}) }^{ (i_{2}, \ldots , i_{k}) }
    \end{array}\right)
    \,]
    }
}_{\text{{\scriptsize $\displaystyle
    \phantom{-\sum_{I \in \mathbb{N}^{k}} \det ( \Psi_{Z} \mid I )}
    \stackrel{\text{Lemma~\ref{Lem:adin}}}{=}
    - \sum_{I \in \mathbb{N}^{k}} \det ( \Psi_{Z} \mid I )
$}}}
}
Z_{i_{k+1}} \cdots Z_{i_{n}}
D_{i_{k+1}}^{*} \cdots D_{i_{n}}^{*} \\
&= 
\Bigg(
{\color{light-gray}
\underbrace{
    {\color{black}
    \binom{n-1}{k} + \binom{n-1}{k-1}
    }
}_{\text{{\scriptsize $
    \phantom{\binom{n}{k}}
    =
    \binom{n}{k}
$}}}
}
\Bigg)
\Bigg(
\sum_{I \in \mathbb{N}^{k}}
\det ( \Psi \mid I )
\Bigg)
\Bigg(
{\color{light-gray}
\underbrace{
    {\color{black}
    \sum_{i_{k+1}, \ldots , i_{n}}
    Z_{i_{k+1}} Z_{i_{k+2}} \cdots Z_{i_{n}}
    D_{i_{k+1}}^{*}
    D_{i_{k+2}}^{*} \cdots D_{i_{n}}^{*}
    }
}_{\text{{\scriptsize $\displaystyle
    \phantom{: (D_{Z}^{*})^{n-k} :}
    \;=\;
    : (D_{Z}^{*})^{n-k} :
$}}}
}
\Bigg)
\end{split}
\end{equation*}
Here, in the final step of the transformation, we used Lemma~\ref{Lem:adin} given below.

By substituting the above computations into (\ref{eq:RHS}), we obtain the following.
\begin{equation*}
\begin{split}
n! \cdot (\text{RHS of (\ref{eq:a_n})})
&\;=\;
n ! \cdot (\text{RHS of (\ref{eq:RHS})}) \\
&\;=\; 
:(D_{Z}^{*})^{n}:
+
\sum_{k=1}^{n-1}
\binom{n}{k}
\left(
    \sum_{I \in \mathbb{N}^{k}}
    \det ( \Psi_{Z} \mid I )
\right)
:(D_{Z}^{*})^{n-k}:
+ \sum_{I \in \mathbb{N}^{n}}
    \det ( \Psi_{Z} \mid I ) \\
&\;=\;
\sum_{k=1}^{n}
\binom{n}{k}
\left(
    \sum_{I \in \mathbb{N}^{k}}
    \det ( \Psi_{Z} \mid I )
\right)
:(D_{Z}^{*})^{n-k}:
\;=\;
n! \cdot a_{n}
\end{split}
\end{equation*}
\end{proof} 

To complete the proof of Lemma~\ref{Lem:an_decomp} above, it remains to show the following.

\medskip

\begin{Lem} 
\label{Lem:adin}
For any $k \in \mathbb{N}$, we have
\begin{equation*}
\sum _{i_{1},\cdots ,i_{k}}
\Big[
    D_{i_{1}}^{*},
    \det
    \left( \begin{array}{ccc}
    Z_{i_{1}} & \cdots & Z_{i_{k}} \\
    \Psi_{i_{2},i_{1}} & \cdots & \Psi_{i_{2},i_{k}} \\
    \vdots & \cdots & \vdots \\
    \Psi_{i_{k},i_{1}} & \cdots & \Psi_{i_{k},i_{k}}
    \end{array} \right)
\Big]
=
\sum _{I \in \mathbb{N}^{k}}
\mathrm{det} ( \Psi_{Z} \mid I ).
\end{equation*}
\end{Lem} 
\begin{proof}[Proof of Lemma~\ref{Lem:adin}]
First, by the Leibniz rule for the commutator as a Lie bracket, the following identity holds.
\begin{equation*}
\begin{split}
&
\Big[
    D_{i_{1}}^{*},
    \det
    \left(\begin{array}{ccc}
    Z_{i_{1}} & \cdots & Z_{i_{k}} \\
    \Psi_{i_{2},i_{1}} & \cdots & \Psi_{i_{2},i_{k}} \\
    \vdots & \cdots & \vdots \\
    \Psi_{i_{k},i_{1}} & \cdots & \Psi_{i_{k},i_{k}}
    \end{array}\right)
\Big] \\
&=
{\color{light-gray}
\underbrace{
    {\color{black}
    \det
    \left(\begin{array}{ccc}
    [\, D_{i_{1}}^{*}, Z_{i_{1}} \,]
    & \cdots
    & [\, D_{i_{1}}^{*}, Z_{i_{k}} \,] \\
    \Psi_{i_{2},i_{1}} & \cdots & \Psi_{i_{2},i_{k}} \\
    \vdots & \cdots & \vdots \\
    \Psi_{i_{k},i_{1}} & \cdots & \Psi_{i_{k},i_{k}}
    \end{array}\right)
    }
}_{\text{{\scriptsize $\displaystyle
    \phantom{
        \det ( \Psi_{Z} \mid (i_{1}, \cdots, i_{k}) )
    }
    =
    \det ( \Psi_{Z} \mid (i_{1}, \cdots, i_{k}) )
$}}}
}
+ 
\sum_{l=2}^{k}
\det
\left(\begin{array}{ccc}
Z_{i_{1}} & \cdots & Z_{i_{l}} \\
\Psi_{i_{2},i_{1}} & \cdots & \Psi_{i_{2},i_{k}} \\
\vdots & \cdots & \vdots \\
{[\, D_{i_{1}}^{*}, \Psi_{i_{l},i_{1}} \,]}
& \cdots
& {[\, D_{i_{1}}^{*}, \Psi_{i_{l},i_{k}} \,]} \\
\vdots & \cdots & \vdots \\
\Psi_{i_{k},i_{1}} & \cdots & \Psi_{i_{k},i_{k}} 
\end{array}\right) .
\end{split}
\end{equation*}
Therefore, it suffices to show that the second term on the right-hand side vanishes after summing over $i_{1}, i_{2}, \ldots , i_{k}$.

To this end, for each summand in the second term $\sum_{l=2}^{k}(...)$, we fix $l=2,\cdots ,k$ arbittrarily, and consider the sum taken in particular over $i_{1}$ and $i_{l}$. We note that $[D_{i}, R] \subset R$, and that, by a change of variables, the following identity holds.
\begin{equation*}
\sum_{i_{1},i_{l}}
\det
\left(\begin{array}{ccc}
Z_{i_{2}} & \cdots & Z_{i_{l}} \\
\Psi_{i_{2},i_{1}} & \cdots & \Psi_{i_{2},i_{k}} \\
\vdots & \cdots & \vdots \\
{[\, D_{i_{1}}^{*}, \Psi_{i_{l},i_{1}} \,]}
& \cdots
& {[\, D_{i_{1}}^{*}, \Psi_{i_{l},i_{k}} \,]} \\
\vdots & \cdots & \vdots \\
\Psi_{i_{k},i_{1}} & \cdots & \Psi_{i_{k},i_{k}} \\
\end{array}\right)
=
\sum _{i_{1},i_{l}}
\det
\left(\begin{array}{ccc}
Z_{i_{1}} & \cdots & Z_{i_{k}} \\
{[\, D_{i_{1}}^{*}, \Psi_{i_{2},i_{1}} \,]}
& \cdots
& {[\, D_{i_{1}}^{*}, \Psi_{i_{2},i_{k}} \,]} \\
\Psi_{i_{3},i_{1}} & \cdots & \Psi_{i_{3},i_{k}} \\
\vdots & \cdots & \vdots \\
\Psi_{i_{k},i_{1}} & \cdots & \Psi_{i_{k},i_{k}}
\end{array}\right)
\end{equation*}
Therefore, it suffices to show the following for the expression obtained by summing this right-hand side over the remaining variables $
\widehat{i_{1}}, i_{2}, \ldots , \widehat{i_{l}}, \ldots , i_{k}
$.
\begin{equation}
\label{eq:lem_key}
\sum _{i_{1},\cdots ,i_{k}}
\det
\left(\begin{array}{ccc}
Z_{i_{1}} & \cdots & Z_{i_{k}} \\
{[\, D_{i_{1}}^{*}, \Psi_{i_{2},i_{1}} \,]}
& \cdots
& {[\, D_{i_{1}}^{*}, \Psi_{i_{2},i_{k}} \,]} \\
\Psi_{i_{3},i_{1}}
& \cdots
& \Psi_{i_{3},i_{k}} \\
\vdots & \cdots & \vdots \\
\Psi_{i_{k},i_{1}} 
& \cdots
& \Psi_{i_{k},i_{k}}
\end{array} \right)
= 0.
\end{equation}

By the Jacobi identity for the commutator, the following holds.
\begin{equation*}
\begin{split}
[\, D_{i}^{*}, \Psi_{j,k} \,]
&= [\, D_{i}^{*}, [\, D_{j}^{*}, Z_{k} \,] \,] \\
&=
- [\, D_{j}^{*},
\hspace{-8mm}
{\color{light-gray}
\underbrace{
    {\color{black}
    [\, Z_{k}, D_{i}^{*} \,]
    }
}_{\text{{\scriptsize $
    \phantom{- [\, D_{i}^{*}, Z_{k} \,]}
    =
    - [\, D_{i}, Z_{k} \,]
$}}}
}
\hspace{-8mm}
\,]
- [\, Z_{k},
{\color{light-gray}
\underbrace{
    {\color{black}
    [\, D_{i}^{*}, D_{j}^{*} \,]
    }
}_{\text{{\scriptsize $
    \phantom{0} = 0
$}}}
}
\,]
=
[\, D_{j}^{*}, [\, D_{i}^{*}, Z_{k} \,] \,]
=
[\, D_{j}^{*}, \Psi_{i,k} \,]
\end{split}
\end{equation*}
From this, the following holds for each summand on the left-hand side of (\ref{eq:lem_key}).
\begin{equation}
\label{eq:smd}
\det
\left(\begin{array}{ccc}
Z_{i_{1}} & \cdots & Z_{i_{k}} \\
{[\, D_{i_{1}}^{*}, \Psi_{i_{2},i_{1}} \,]}
& \cdots
& {[\, D_{i_{1}}^{*}, \Psi_{i_{2},i_{k}} \,]} \\
\Psi_{i_{3},i_{1}} & \cdots & \Psi_{i_{3},i_{k}} \\
\vdots & \cdots & \vdots \\
\Psi_{i_{k},i_{1}} & \cdots & \Psi_{i_{k},i_{k}}
\end{array}\right)
=
\det
\left(\begin{array}{ccc}
Z_{i_{1}} & \cdots & Z_{i_{k}} \\
{[\, D_{i_{2}}^{*}, \Psi_{i_{1},i_{1}} \,]}
& \cdots
& {[\, D_{i_{2}}^{*}, \Psi_{i_{1},i_{k}} \,]} \\
\Psi_{i_{3},i_{1}} & \cdots & \Psi_{i_{3},i_{k}} \\
\vdots & \cdots & \vdots \\
\Psi_{i_{k},i_{1}} & \cdots & \Psi_{i_{k},i_{k}}
\end{array}\right)
\end{equation}
Also, by interchanging $i_{1}$ and $i_{2}$, we obtain the following.
\begin{equation*}
\begin{split}
&
\left(\begin{array}{c}
\text{The expression obtained by interchanging} \\
\text{the roles of $i_{1}$ and $i_{2}$ on LHS of (\ref{eq:smd})}
\end{array}\right) \\
&=
\det
\left(\begin{array}{ccccc}
Z_{i_{2}} & Z_{i_{1}} & Z_{i_{3}} & \cdots & Z_{i_{k}} \\
{[\, D_{i_{2}}^{*}, \Psi_{i_{1},i_{2}} \,]}
& {[\, D_{i_{2}}^{*}, \Psi_{i_{1},i_{1}} \,]}
& {[\, D_{i_{2}}^{*}, \Psi_{i_{1},i_{3}} \,]}
& \cdots
& {[\, D_{i_{2}}^{*}, \Psi_{i_{1},i_{k}} \,]} \\
\Psi_{i_{3},i_{2}} & \Psi_{i_{3},i_{1}} & \Psi_{i_{3},i_{3}} & \cdots & \Psi_{i_{3},i_{k}} \\
\vdots & \vdots & \vdots & \cdots & \vdots \\
\Psi_{i_{k},i_{2}} & \Psi_{i_{k},i_{1}} & \Psi_{i_{k},i_{3}} & \cdots & \Psi_{i_{k},i_{k}}
\end{array}\right) \\
&=
-
\det
\left(\begin{array}{ccc}
Z_{i_{1}} &  \cdots & Z_{i_{k}} \\
{[\, D_{i_{2}}^{*}, \Psi_{i_{1},i_{1}} \,]}
& \cdots
& {[\, D_{i_{2}}^{*}, \Psi_{i_{1},i_{k}} \,]} \\
\Psi_{i_{3},i_{1}} & \cdots & \Psi_{i_{3},i_{k}} \\
\vdots & \cdots & \vdots \\
\Psi_{i_{k},i_{1}} & \cdots & \Psi_{i_{k},i_{k}}
\end{array}\right)
=
(\text{RHS of (\ref{eq:smd})})
\end{split}
\end{equation*}
From the above, we have
$
\sum_{i_{1}, i_{2}, \ldots , i_{k}} (\text{LHS of (\ref{eq:smd})})
=
- \sum_{i_{1}, i_{2}, \ldots , i_{k}} (\text{LHS of (\ref{eq:smd})})
$
which implies (\ref{eq:lem_key}).
\end{proof} 

\begin{proof}[Proof of Lemma~\ref{Lem:cn_decomp}]

Applying the cofactor expansion along the first row to the summand in $\sum_{i_{1}, i_{2}, \ldots , i_{n+1}}(...)$, which appears in the definition of $c_{n}$, and arranging the terms using the same technique as in the proof of Lemmas~\ref{Lem:an_decomp}, we obtain the following.
\begin{equation*}
\begin{split}
&
\sum_{i_{1}, \cdots, i_{n+1}}
D_{i_{1}}^{*}
\mathrm{det}
\left(\begin{array}{ccc}
Z_{i_{1}} & \cdots & Z_{i_{n+1}} \\
\Psi_{i_{2},i_{1}} & \cdots & \Psi_{i_{2},i_{n+1}} \\
\vdots & \cdots & \vdots \\
\Psi_{i_{n+1},i_{1}} & \cdots & \Psi_{i_{n+1},i_{n+1}}
\end{array}\right)\\
&=
{\color{light-gray}
\underbrace{
    {\color{black}
    \sum_{i_{1}} D_{i_{1}}^{*} Z_{i_{1}}
    }
}_{\text{{\scriptsize $\displaystyle
    \phantom{\mathrm{Tr}( \Phi_{Z} )}
    =
    \mathrm{Tr}( \Phi_{Z} )
$}}}
}
\sum_{I \in \mathbb{N}^{n} }
\det ( \Psi_{Z} \mid I )
-
n
\underbrace{
    \sum_{i_{1},\cdots ,i_{n+1}}
    \Phi_{i_{1}i_{2}}
    \mathrm{det}
    \Psi_{(i_{1}, \widehat{i_{2}}, i_{3}, \ldots , i_{n+1})}^{ (i_{2}, i_{3}, \ldots , i_{n+1}) }
}_{\text{{\small $\displaystyle
    \phantom{X_{2}} = X_{2}
$}}}.
\end{split}
\end{equation*}
Here, we set $X_{2}$ as above. Applying the cofactor expansion along the first row to the determinant appearing $X_{2}$, and arranging the terms using the same technique as in the proof of Lemmas~\ref{Lem:an_decomp}, we obtain the following transformation.
\begin{equation*}
\begin{split}
X_{2}
&=
\sum_{i_{1},\cdots ,i_{n+1}}
\Phi_{i_{1}i_{2}}
\mathrm{det}
\Psi_{(i_{1}, i_{3}, \ldots , i_{n+1})}^{ (i_{2}, i_{3},  \ldots , i_{n+1}) } \\
&=
{\color{light-gray}
\underbrace{
    {\color{black}
    \sum_{i_{1}, i_{2}}
    \Psi_{i_{1}, i_{2}} \Psi_{i_{2}, i_{1}}
    }
}_{\text{{\scriptsize $\displaystyle
    \phantom{\mathrm{Tr}( \Phi_{Z} \Psi_{Z} )}
    =
    \mathrm{Tr}( \Phi_{Z} \Psi_{Z} )
$}}}
}
\sum_{I \in \mathbb{N}^{n-1}}
\det ( \Psi_{Z} \mid I )
-
(n-1)
\underbrace{
    \sum_{i_{1},\cdots ,i_{n+1}}
    \Phi_{i_{1},i_{2}}
    \Psi_{i_{2},i_{3}}
    \mathrm{det}
    \Psi_{(i_{1}, i_{4}, \ldots , i_{n+1})}^{ (i_{3}, i_{4}, \ldots , i_{n+1}) }
}_{\text{{\small $\displaystyle
    \phantom{X_{3}} = X_{3}
$}}}
\end{split}
\end{equation*}
Here, we define $X_{3}$ as above. Similarly, for $k=2,3,\ldots , n$, we define $X_{k}$ by
\begin{equation*}
X_{k}
=
\sum_{i_{1},\cdots ,i_{n+1}}
\Phi_{i_{1},i_{2}}
\Psi_{i_{2},i_{3}}
\cdots
\Psi_{i_{k-1},i_{k}}
{\color{light-gray}
\underbrace{
    {\color{black}
    \mathrm{det}
    \left( \begin{array}{cccc}
    \Psi_{i_{k},i_{1}} & \Psi_{i_{k},i_{k+1}} & \cdots & \Psi_{i_{k},i_{n+1}} \\
    \Psi_{i_{k+1},i_{1}} & \Psi_{i_{k+1},i_{k+1}} & \cdots & \Psi_{i_{k+1},i_{n+1}} \\
    \vdots & \vdots & \cdots & \vdots \\
    \Psi_{i_{n+1},i_{1}} & \Psi_{i_{n+1},i_{k+1}} & \cdots & \Psi_{i_{n+1},i_{n+1}}
    \end{array} \right)
    }
}_{\text{{\scriptsize $\displaystyle
    \phantom{ \det \Psi_{ ( i_{1}, i_{k+1}, \ldots , i_{n+1} ) }^{ ( i_{k}, i_{k+1}, \ldots , i_{n+1} ) } }
    =
    \det \Psi_{ ( i_{1}, i_{k+1}, \ldots , i_{n+1} ) }^{ ( i_{k}, i_{k+1}, \ldots , i_{n+1} ) }
$}}}
}.
\end{equation*}
Then, it follows that the family $\{ X_{k} \}$ satisfies the following recurrence relation.
\begin{equation}
\label{X_k}
X_{k}
=
\mathrm{Tr}( \Phi_{Z} \Psi_{Z}^{k-1})
\sum_{I \in \mathbb{N}^{n+1-k} }
\det ( \Psi_{Z} \mid I )
-
(n-k) X_{k+1}
\end{equation}
The terminal condition is given by the following computation.
\begin{equation*}
\begin{split}
X_{n}
&=
\sum_{i_{1},\cdots ,i_{n+1}}
\Phi_{i_{1},i_{2}}
\Psi_{i_{2},i_{3}}
\cdots
\Psi_{i_{n-1},i_{n}}
{\color{light-gray}
\underbrace{
    {\color{black}
    \mathrm{det}
    \left(\begin{array}{cc}
    \Psi_{i_{n},i_{1}} & \Psi_{i_{n},i_{n+1}} \\
    \Psi_{i_{n+1},i_{1}} & \Psi_{i_{n+1},i_{n+1}}
    \end{array} \right)
    }
}_{\text{{\scriptsize $\displaystyle
    \begin{array}{c}
    \text{$
    \phantom{\Psi_{i_{n},i_{1}} \Psi_{i_{n+1},i_{n+1}}}
    =
    \Psi_{i_{n},i_{1}} \Psi_{i_{n+1},i_{n+1}}
    $} \\
    \text{$
    \phantom{aa-\Psi_{i_{n},i_{n+1}} \Psi_{i_{n+1},i_{1}}}
    -
    \Psi_{i_{n},i_{n+1}} \Psi_{i_{n+1},i_{1}}
    $}
    \end{array}
$}}}
} \\
&=
{\color{light-gray}
\underbrace{
    {\color{black}
    \sum_{i_{1},\cdots ,i_{n+1}}
    \Phi_{i_{1},i_{2}}
    \Psi_{i_{2},i_{3}}
    \cdots
    \Psi_{i_{n-1},i_{n}}
    \Psi_{i_{n},i_{1}}
    \Psi_{i_{n+1},i_{n+1}}
    }
}_{\text{{\scriptsize $
    \phantom{ \mathrm{Tr}( \Phi_{Z} \Psi_{Z}^{n-1} )
    \mathrm{Tr}( \Psi_{Z} ) }
    =
    \mathrm{Tr}( \Phi_{Z} \Psi_{Z}^{n-1} )
    \mathrm{Tr}( \Psi_{Z} )
$}}}
}
-
{\color{light-gray}
\underbrace{
    {\color{black}
    \sum_{i_{1},\cdots ,i_{n+1}}
    \Phi_{i_{1},i_{2}}
    \Psi_{i_{2},i_{3}}
    \cdots
    \Psi_{i_{n-1},i_{n}}
    \Psi_{i_{n},i_{n+1}}
    \Psi_{i_{n+1},i_{1}}
    }
}_{\text{{\scriptsize $\displaystyle
    \phantom{ \mathrm{Tr}( \Phi_{Z} \Psi_{Z}^{n} ) }
    =
    \mathrm{Tr}( \Phi_{Z} \Psi_{Z}^{n} )
$}}}
} \\
&=
\mathrm{Tr}( \Phi_{Z} \Psi_{Z}^{n-1} )
\hspace{-14mm}
{\color{light-gray}
\underbrace{
    {\color{black}
    \mathrm{Tr}( \Psi_{Z} )
    }
}_{\text{{\scriptsize $\displaystyle
    \phantom{
        \sum_{i \in \mathbb{N}} \det ( \Psi_{Z} \mid i )
    }
    =
    \sum_{i \in \mathbb{N}}
    \det ( \Psi_{Z} \mid i )
$}}}
}
\hspace{-14mm}
-
\mathrm{Tr}( \Phi_{Z} \Psi_{Z}^{n} ) .
\end{split}
\end{equation*}

Thus, we obtain the following.
\begin{equation*}
\begin{split}
&
\sum_{i_{1}, \cdots, i_{n+1}}
D_{i_{1}}^{*}
\mathrm{det}
\left(\begin{array}{ccc}
Z_{i_{1}} & \cdots & Z_{i_{n+1}} \\
\Psi_{i_{2},i_{1}} & \cdots & \Psi_{i_{2},i_{n+1}} \\
\vdots & \cdots & \vdots \\
\Psi_{i_{n+1},i_{1}} & \cdots & \Psi_{i_{n+1},i_{n+1}}
\end{array}\right) \\
&=
\mathrm{Tr}( \Phi_{Z} )
\sum_{I \in \mathbb{N}^{n} }
\det ( \Psi_{Z} \mid I )
-
n \cdot
\hspace{-57mm}
{\color{light-gray}
\underbrace{
    {\color{black}
    X_{2}
    }
}_{\text{{\scriptsize $\displaystyle
    \phantom{
        \mathrm{Tr} ( \Phi_{Z} \Psi_{Z} ) \sum_{I \in \mathbb{N}^{n-1} } \det ( \Psi_{Z} \mid I ) - (n-1) X_{3}
    }
    =
    \mathrm{Tr} ( \Phi_{Z} \Psi_{Z} )
    \sum_{I \in \mathbb{N}^{n-1} } \det ( \Psi_{Z} \mid I )
    -
    (n-1) X_{3}
$}}}
} \\
&=
\mathrm{Tr}( \Phi_{Z} )
\sum_{I \in \mathbb{N}^{n} }
\det ( \Psi_{Z} \mid I ) 
-
n
\cdot
\mathrm{Tr} ( \Phi_{Z} \Psi_{Z} )
\sum_{I \in \mathbb{N}^{n-1} } \det ( \Psi_{Z} \mid I ) \\
&\hspace{25mm}+
n (n-1) \cdot
\hspace{-57mm}
{\color{light-gray}
\underbrace{
    {\color{black}
    X_{3}
    }
}_{\text{{\scriptsize $\displaystyle
    \phantom{
        \mathrm{Tr} ( \Phi_{Z} \Psi_{Z}^{2} ) \sum_{I \in \mathbb{N}^{n-2} } \det ( \Psi_{Z} \mid I ) - (n-2) X_{4}
    }
    =
    \mathrm{Tr} ( \Phi_{Z} \Psi_{Z}^{2} )
    \sum_{I \in \mathbb{N}^{n-2} } \det ( \Psi_{Z} \mid I )
    -
    (n-2) X_{4}
$}}}
} \\
&= \cdots \\
&=
\sum_{k=0}^{n-2}
\mathrm{Tr}( \Phi_{Z} (-\Psi_{Z})^{k} )
\frac{ n! }{ (n-k)! }
\sum_{I \in \mathbb{N}^{n-k} }
\det ( \Psi_{Z} \mid I ) \\
&\hspace{20mm}
+
(-1)^{n-1} \cdot
n (n-1) \cdots 2 \cdot
\hspace{-57mm}
{\color{light-gray}
\underbrace{
    {\color{black}
    X_{n}
    }
}_{\text{{\scriptsize $\displaystyle
    \phantom{
       \mathrm{Tr} ( \Phi_{Z} \Psi_{Z}^{n-1} ) \sum_{i \in \mathbb{N}} \det ( \Psi_{Z} \mid i ) - \mathrm{Tr} ( \Phi_{Z} \Psi_{Z}^{n} )
    }
    =
    \mathrm{Tr} ( \Phi_{Z} \Psi_{Z}^{n-1} )
    \sum_{i \in \mathbb{N}}
    \det ( \Psi_{Z} \mid i )
    -
    \mathrm{Tr} ( \Phi_{Z} \Psi_{Z}^{n} )
$}}}
} \\
&=
\sum_{k=0}^{n}
\mathrm{Tr}( \Phi_{Z} (-\Psi_{Z})^{k} )
\frac{ n! }{ (n-k)! }
\sum_{I \in \mathbb{N}^{n-k} }
\det ( \Psi_{Z} \mid I )
\end{split}
\end{equation*}
Therefore, the claim is proved.
\end{proof} 

\subsection{Proof of Theorem \ref{Thm:MB}}

Throughout this section, we assume that $D_{i}(1)=0$ for all $i$. In this case, as a multiplication operator, we have
\begin{equation}
\label{eq:1st_chaos}
D_{i} + D_{i}^{*} = D_{i}^{*}.1.
\end{equation}
Indeed, this follows from
$
D_{i} + D_{i}^{*}
= (D_{i} + D_{i}^{*}).1
= D_{i}.1 + D_{i}^{*}.1
= D_{i}^{*}.1
$.

We begin by establishing the following property of $T_{Z}(t)$, which is characteristic of pull-back type operators.

\medskip

\begin{Lem} 
\label{Lem:LEI}
We have
$
T_{Z}(t).(fg)
=
( T_{Z}(t).f ) \cdot ( T_{Z}(t).g )
$
for $f,g \in R$.
\end{Lem} 
\begin{proof} 
We first note that 
$
(D_{i}).(fg)
=
( [\, D_{i}, f \,] g ).1
+
( f [\, D_{i}, g \,] ).1
$
for $f$, $g \in R$ by using $D_i (1) = 0$.
In particular, we have $[D_{i}, f].1 = D_{i}.f \in R$.
Thus it holds
\begin{equation*}
\begin{split}
&
T_{Z}(t).(fg)
=
\sum_{n=0}^{\infty}
\frac{ t^{n} }{ n! }
\sum_{i_{1}, \cdots, i_{n}}
Z_{i_{1}} \cdots Z_{i_{n}}
( D_{i_{1}} \cdots D_{i_{n}} ).( fg ) \\
&=
\sum_{n=0}^{\infty}
\frac{ t^{n} }{ n! }
\sum_{i_{1}, \cdots, i_{n}}
Z_{i_{1}} \cdots Z_{i_{n}}
\sum_{k=0}^{n}
\binom{n}{k}
( D_{i_{1}} \cdots D_{i_{k}} ).(f)
\cdot
( D_{i_{k+1}} \cdots D_{i_{n}} ).(g) \\
&=
\sum_{n=0}^{\infty}
\frac{ t^{n} }{ n! }
\sum_{k=0}^{n}
\binom{n}{k}
\Bigg(
{\color{light-gray}
\underbrace{
    {\color{black}
    :
    \sum_{i_{1},\ldots , i_{k}}
    Z_{i_{1}} \cdots Z_{i_{n}}
    D_{i_{1}} \cdots D_{i_{n}}
    :
    }
}_{\text{{\scriptsize $
    \phantom{:(D_{Z})^{k}:}
    \;=\;
    :(D_{Z})^{k}:
$}}}
}
\Bigg).(f) \\
&\hspace{50mm}
\cdot
\Bigg(
{\color{light-gray}
\underbrace{
    {\color{black}
    :
    \sum_{i_{k+1},\ldots , i_{n}}
    Z_{i_{k+1}} \cdots Z_{i_{n}}
    D_{i_{k+1}} \cdots D_{i_{n}}
    :
    }
}_{\text{{\scriptsize $
    \phantom{:(D_{Z})^{n-k}:}
    \;=\;
    :(D_{Z})^{n-k}:
$}}}
}
\Bigg).(g) \\
&=
( T_{Z}(t).f )
\cdot
( T_{Z}(t).g ) .
\end{split}
\end{equation*}
\end{proof} 

Now, we are ready to prove Theorem~\ref{Thm:MB}.

\medskip

\begin{proof}[Proof of Theorem~\ref{Thm:MB}]
We write
$\displaystyle
e_{t}
=
\exp \left(
    \int_{0}^{t}
    : \mathrm{e}^{-t^{\prime} D_{Z}} D_{Z}^{*}:(1)
    \,\mathrm{d}t^{\prime}
\right)
\in R[\![t]\!]
$
and denote by $e_{t}^{(n)}$ its $n$-th derivative.
For $n = 0,1,\ldots$ and $( i_{1}, \cdots, i_{n} ) \in \mathbb{N}^{n}$, we set
\begin{equation*}
P_{n}^{i_{1}, \ldots, i_{n}}
\;=\;
: (D_{i_1}+D_{i_1}^{*}) \cdots (D_{i_{n}}+D_{i_{n}}^{*}) :(1).
\end{equation*}
Then we have
$\displaystyle
\mathcal{E}_{Z}(t).1
=
\sum_{n=0}^{\infty} \frac{t^{n}}{n!}
\sum_{(i_{1}, \ldots , i_{n}) \in \mathbb{N}^{n}}
Z_{i_{1}} \cdots Z_{i_{n}}
P_{n}^{i_{1}, \ldots , i_{n}}
$.

First, for each $n \in \mathbb{N}$ and $(i_{1}, i_{2}, \ldots , i_{n}) \in \mathbb{N}^{n}$, we show that the following identity holds:
\begin{equation}
\label{eq:inductive}
\sum_{i_{1}, \ldots, i_{n}}
[\,
(
    Z_{i_{1}} \cdots Z_{i_{n}}
    :\mathrm{e}^{-tD_{Z}}:
).
( P_{n}^{i_{1}, \ldots, i_{n}})
\,]
\cdot
e_{t}
=
e_{t}^{(n)}.
\end{equation}
Then, by setting $t=0$, we obtain
$\displaystyle
\left.
\frac{\mathrm{d}^{n}}{\mathrm{d}t^{n}}
\right\vert_{t=0}
\mathcal{E}_{Z}(t).1
=
\sum_{(i_{1}, \ldots , i_{n}) \in \mathbb{N}^{n}}
Z_{i_{1}} \cdots Z_{i_{n}}
P_{n}^{i_{1}, \ldots , i_{n}}
=
e_{0}^{(n)}
$,
which proves Theorem~\ref{Thm:MB}.

In the case $n=1$, by the assumption $D_{i}(1) = 0$, we have
$
P_{1}^{i}
\;=\; : D_{i} + D_{i}^{*} : (1)
= ( D_{i} + D_{i}^{*} ) (1)
= D_{i}^{*}.1
$,
and hence the identity follows from the following computation:
\begin{equation*}
\sum_{i \in \mathbb{N}}
[\,
( Z_{i} :\mathrm{e}^{-tD_{Z}}: ).(
\hspace{-5mm}
{\color{light-gray}
\underbrace{
    {\color{black}
    P_{1}^{i}
    }
}_{\text{{\scriptsize $\displaystyle
    \phantom{D_{i}^{*}.1}
    =
    D_{i}^{*}.1
$}}}
}
\hspace{-5mm}
)
\,]
\cdot
e_{t}
=
{\color{light-gray}
\underbrace{
    {\color{black}
    \sum_{i \in \mathbb{N}}
    [\,(
    \hspace{-8mm}
    {\color{light-gray}
    \underbrace{
        {\color{black}
        Z_{i} :\mathrm{e}^{-tD_{Z}}: D_{i}^{*}
        }
    }_{\text{{\scriptsize $\displaystyle
        \phantom{: Z_{i} \mathrm{e}^{-tD_{Z}} D_{i}^{*} :}
        \;=\;
        : Z_{i} \mathrm{e}^{-tD_{Z}} D_{i}^{*} :
    $}}}
    }
    \hspace{-8mm}
    ).1
    \,]
    }
}_{\text{{\scriptsize $\displaystyle
    \phantom{: \mathrm{e}^{-tD_{Z}} D_{Z}^{*} :(1)}
    \;=\;
    : \mathrm{e}^{-tD_{Z}} D_{Z}^{*} :(1)
$}}}
}
\hspace{-2mm}
\cdot e_{t}
=
[\, :\mathrm{e}^{ -t D_{Z} } D_{Z}^{*} : (1) \,]
\cdot
e_{t}
=
e_{t}^{(1)}.
\end{equation*}

Now fix $n \in \mathbb{N}$ and assume that (\ref{eq:inductive}) holds for $n$. Then the following computation yields the result.
\begin{equation}
\label{eq:DRV}
\begin{split}
&
e_{t}^{(n+1)}
=
( e_{t}^{(n)} )^{\prime}
=
\hspace{-25mm}
{\color{light-gray}
\underbrace{
    {\color{black}
    \Bigg(
        \sum_{i_{1}, \ldots, i_{n}}
        (\,
            Z_{i_{1}} \cdots Z_{i_{n}}
            : \mathrm{e}^{-t D_{Z}} :
        \,).
        ( P_{n}^{i_{1}, \ldots , i_{n}} )
    \Bigg)^{\prime}
    }
}_{\text{{\scriptsize $\displaystyle
    \begin{array}{c}
    \displaystyle
    \phantom{MMMMMMMMMMM}
    =
    \sum_{i_{1}, \ldots, i_{n}, i_{n+1}}
    (\,
        Z_{i_{1}} \cdots Z_{i_{n}}
        : Z_{i_{n+1}} (-D_{i_{n+1}}) \mathrm{e}^{-t D_{Z}} :
    \,).
    ( P_{n}^{i_{1}, \ldots, i_{n}} ) \\
    \displaystyle
    \phantom{MMMMMMMMMMM}
    =
    \sum_{i_{1}, \ldots, i_{n}, i_{n+1}}
    (\,
        Z_{i_{1}} \cdots Z_{i_{n}} Z_{i_{n+1}}
        : \mathrm{e}^{-t D_{Z}} :
    \,).
    ( -D_{i_{n+1}}.P_{n}^{i_{1}, \ldots , i_{n}} )
    \end{array}
$}}}
}
\hspace{-25mm}
\cdot e_{t} \\
&\hspace{35mm}
+
\left(
\sum_{i_{1}, \cdots, i_{n}}
\Big(
Z_{i_{1}} \cdots Z_{i_{n}}
: \mathrm{e}^{ -t D_{Z} } : 
\Big).
\big( P_{n}^{i_{1}, \ldots , i_{n}} \big)
\right)
\cdot
\hspace{-25mm}
{\color{light-gray}
\underbrace{
    {\color{black}
    e_{t}^{(1)}
    }
}_{\text{{\scriptsize $\displaystyle
    \begin{array}{c}
    \rotatebox{90}{$=$} \\
    \displaystyle
    \left(
    \sum_{i_{n+1}}
    (\, Z_{i_{n+1}} :\mathrm{e}^{-tD_{Z}}: \,).(D_{i_{n+1}}^{*}.1)
    \right)
    \cdot e_{t}
    \end{array}
$}}}
} \\
&=
\Bigg(\sum_{i_{1}, \ldots , i_{n}, i_{n+1}}
(\,
    Z_{i_{1}} \cdots Z_{i_{n}} Z_{i_{n+1}}
    :\mathrm{e}^{-tD_{Z}}:
\,).\big(
    -D_{i_{n+1}}.P_{n}^{i_{1},\ldots , i_{n}}
\big)
\Bigg) \cdot e_{t} \\
&\hspace{10mm}+
\Bigg(
\sum_{i_{1}, \ldots , i_{n}, i_{n+1}}
(\,
    Z_{i_{1}} \cdots Z_{i_{n}} Z_{i_{n+1}}
\,).\Big(
    {\color{light-gray}
    \underbrace{
        {\color{black}
        [\, :\mathrm{e}^{-tD_{Z}}:( P_{n}^{i_{1}, \ldots , i_{n}} ) \,]
        \cdot
        [\, :\mathrm{e}^{-tD_{Z}}:( D_{i_{n+1}}^{*}.1 ) \,]
        }
    }_{\text{{\scriptsize $\displaystyle
        \begin{array}{c}
        \text{
        \phantom{Lemma~\ref{Lem:LEI}}
        \rotatebox{90}{$=$}
        Lemma~\ref{Lem:LEI}
        } \\
        :\mathrm{e}^{-tD_{Z}}:( P_{n}^{i_{1}, \ldots , i_{n}} \cdot ( D_{i_{n+1}}^{*}.1 ) )
        \end{array}
    $}}}
    }
\Big)
\Bigg) \cdot e_{t} \\
&=
\Bigg(
\sum_{i_{1}, \ldots , i_{n}, i_{n+1}}
(\,
    Z_{i_{1}} \cdots Z_{i_{n}} Z_{i_{n+1}}
    :\mathrm{e}^{-tD_{Z}}:
\,).\Big(
    \underbrace{
        - D_{i_{n+1}}.P_{n}^{i_{1}, \ldots , i_{n}}
        +
        P_{n}^{i_{1}, \ldots , i_{n}} \cdot ( D_{i_{n+1}}^{*}.1 )
    }_{\text{{\small $\displaystyle
        \phantom{Q}
        =
        Q
    $}}}
\Big)
\Bigg) \cdot e_{t}
\end{split}
\end{equation}
Here, we define $Q$ as above. Then, by the following computation, we see that $Q$ is equal to $P_{n+1}^{i_{1}, \ldots , i_{n+1}}$:
\begin{equation*}
\begin{split}
Q
&=
- D_{i_{n+1}}.P_{n}^{i_{1}, \ldots , i_{n}}
+
\hspace{-10mm}
{\color{light-gray}
\underbrace{
    {\color{black}
    ( D_{i_{n+1}}^{*}.1 ) \times
    }
}_{
    \phantom{ D_{i_{n+1}} + D_{i_{n+1}}^{*} }
    \stackrel{(\ref{eq:1st_chaos})}{=}
    D_{i_{n+1}} + D_{i_{n+1}}^{*}
}
}
\hspace{-10mm}
P_{n}^{i_{1}, \ldots , i_{n}} \\
&=
D_{i_{n+1}}^{*}
\hspace{-43mm}
{\color{light-gray}
\underbrace{
    {\color{black}
    P_{n}^{i_{1}, \ldots , i_{n}}
    }
}_{\text{{\scriptsize $\displaystyle
    \phantom{: (D_{i_1}+D_{i_1}^{*}) \cdots (D_{i_{n}}+D_{i_{n}}^{*}) :(1)}
    \;=\;
    : (D_{i_1}+D_{i_1}^{*}) \cdots (D_{i_{n}}+D_{i_{n}}^{*}) :(1)
$}}}
}
\hspace{-35mm}
+
\hspace{5mm}
: (D_{i_1}+D_{i_1}^{*}) \cdots (D_{i_{n}}+D_{i_{n}}^{*}) :
(
{\color{light-gray}
\overbrace{
    {\color{black}
    D_{i_{n+1}}.1
    }
}^{\text{{\scriptsize $\displaystyle
    0 = \phantom{0}
$}}}
}
) \\
&=
: (D_{i_1}+D_{i_1}^{*}) \cdots (D_{i_{n}}+D_{i_{n}}^{*}) (D_{i_{n+1}}+D_{i_{n+1}}^{*}) : (1)
=
P_{n+1}^{i_{1}, \ldots , i_{n}, i_{n+1}}
\end{split}
\end{equation*}
Therefore, (\ref{eq:inductive}) holds, and hence Theorem~\ref{Thm:MB} is proved.
\end{proof} 

\section{Reduction to the classical Ramer--Kusuoka formula}

In this section, we show that when our algebra is `represented' on the Wiener space, Theorem~\ref{Thm:main} yields the classical Ramer--Kusuoka formula. To this end, we begin with the fundamental Hermite polynomials.

The Hermite polynomials $H_{n}$, $n \in \mathbb{Z}_{\geq 0}$ is defined by $H_{0}(x) = 1$ and $H_{n}(x) = \partial^{*n} 1(x)$ for $n \in \mathbb{N}$ and $x \in \mathbb{R}$, where
\begin{equation*}
\partial^{*} f(x) := - f^{\prime}(x) + x f(x),
\quad x \in \mathbb{R}
\end{equation*}
for any differentiable function $f(x)$.
It is well known that the Hermite polynomials satisfy the following relations:
\begin{equation}
\label{eq:Hermite}
H_{n}^{\prime} (x) = n H_{n-1} (x)
\quad
\text{and}
\quad
H_{n}(x) = x H_{n-1}(x) - (n-1) H_{n-2}(x)
\end{equation}
These Hermite polynomials are characterized by the equation
\begin{equation*}
H_{n}(x)
=
(-1)^{n}
\mathrm{e}^{x^2/2}
\frac{ \mathrm{d}^{n} }{ \mathrm{d}x^{n} }
\mathrm{e}^{ -x^{2}/2 }.
\end{equation*}

\subsection{Representation on Wiener space}

Let $W = \{ w \in C([0,1] \to \mathbb{R}): w(0) = 0 \}$, and consider the triple $( W, \mathcal{B}(W), \mu )$, where $\mathcal{B}(W)$ denotes the Borel $\sigma$-algebra on $W$ and $\mu$ is the Wiener measure; this is the classical Wiener space. The $L_{2}$-inner product on the classical Wiener space is given by
$
\langle
    F, G
\rangle_{\mu}
=
\int_{W} F(w) G(w) \,\mu (\mathrm{d}w)
$.

Let $H$ be the Cameron--Martin subspace of $W$, that is, the set of elements of $W$ that are absolutely continuous and whose derivatives are square-integrable. Then, $H$, equipped with the inner product
$
\langle h, k \rangle_{H}
=
\int_{0}^{1} \dot{h}(s) \dot{k}(s) \,\mathrm{d}s
$,
$g,h \in H$
forms a Hilbert space. Fix once and for all a complete orthonormal basis $\{ h_{i} \}_{i=1}^{\infty}$ of $H$.

We define
$\Gamma^{\mathrm{alg}}( \int \dot{H} \,\mathrm{d}w )$
as the space of polynomial functionals on the Wiener space; that is, the set of all random variables of the form
\begin{equation*}
F(w)
=
f(
\int_{0}^{1} \dot{k}_{1}(s) \,\mathrm{d}w(s),
\ldots ,
\int_{0}^{1} \dot{k}_{n}(s) \,\mathrm{d}w(s)
),
\end{equation*}
where $n \in \mathbb{N}$, $f(x_{1}, \ldots , x_{n})$ is a polynomial in $n$ variables, and $k_{1}, \ldots , k_{n} \in H$. Then
$\Gamma^{\mathrm{alg}}( \int \dot{H} \,\mathrm{d}w )$
is a commutative algebra over $k = \mathbb{R}$.

For $h \in H$, let $D_{h}$ denote the derivative in the direction of $h$, acting on $\Gamma^{\mathrm{alg}}( \int \dot{H} \,\mathrm{d}w )$:
\begin{equation}
\label{eq:def_D}
D_{h} (F) (w)
=
\left.
\frac{\mathrm{d}}{\mathrm{d}\varepsilon}
\right\vert_{\varepsilon = 0}
F( w + \varepsilon \cdot h ),
\quad
w \in W
\end{equation}
and define
\begin{equation*}
D_{h}^{*} (F) (w)
=
- ( D_{h} F ) (w)
+
\left( \int_{0}^{1} \dot{h}(s) \mathrm{d}w(s) \right)
F(w), 
\quad
w \in W
\end{equation*}
Hereafter we sometimes write $D_{h}.F$ and $D_{h}^{*}.F$ in place of $D_{h}(F)$ and $D_{h}^{*}(F)$, respectively.

The Heisenberg algebra, a subalgebra in $\mathcal{D}(R)$, generated by $\{ D_{i}, D_{i}^{*} \}_{i}$ can be represented as
\begin{equation*}
D_{i} = D_{h_{i}},
\quad
D_{i}^{*} = D_{h_{i}}^{*},
\quad
i \in \mathbb{N}.
\end{equation*}
Then it is easily verified that they satisfy the CCR:
\begin{equation}
\label{eq:CCR}
[ D_{i}, D_{j} ] = [ D_{i}^{*}, D_{j}^{*} ] = 0,
\quad
\text{and}
\quad
[ D_{i}, D_{j}^{*}]
=
\delta_{i,j}. 
\end{equation}
Therefore, 
$\Gamma^{\mathrm{alg}}( \int \dot{H} \,\mathrm{d}w )$
is the algebraic (Bosonic-) Fock space generated by $\{ D_{i}, D_{i}^{*} \}_{i}$. By taking the completion, one obtains that $L_{2}( W, \mathcal{B}(W), \mu )$ is a (Bosonic) Fock space.

\begin{Prop} 
\label{Prop:basics}
Let $h \in H$ and $F,G \in \Gamma^{\mathrm{alg}}(\int \dot{H} \,\mathrm{d}w)$. Then the following hold:
\begin{itemize}
\item[(i)]
$\Vert h \Vert_{H} = 1$
$\Rightarrow$
$\displaystyle
D_{h}^{*n}.1(w) = H_{n}( \int_{0}^{1}\dot{h}(s) \,\mathrm{d}w(s) )
$.

\medskip
\item[(ii)]
$
D_{h}^{*}(FG) = ( D_{h}^{*}.F) \cdot G - F \cdot D_{h}.G
$.

\medskip
\item[(iii)]
$\displaystyle
\langle D_{h} F, G \rangle_{\mu}
=
\langle F , D_{h}^{*}G \rangle_{\mu}
$.
\end{itemize}
\end{Prop} 

In particular, note that the following additional property holds, which cannot be expressed solely in terms of the commutation relations of the algebra:
For
$
F,G
\in
\Gamma^{\mathrm{alg}}( \textstyle{\int \dot{H} \,\mathrm{d}w} )
$
and
$i \in \mathbb{N}$,
\begin{equation*}
\langle D_{i} F, G \rangle_{\mu}
=
\langle F , D_{i}^{*}G \rangle_{\mu},
\quad
\int_{W} D_{h}^{*} F (w) \,\mu (\mathrm{d}w) = 0.
\end{equation*}

\begin{proof}[Proof of Proposition~\ref{Prop:basics}]
(i) follows by mathematical induction, using the definition of $D_{h}^{*}$ and property (\ref{eq:Hermite}) of the Hermite polynomials.
(ii) is immediate from the definition of $D_{h}^{*}$ together with the Leibniz rule satisfied by $D_{h}$.
(iii) It suffices to prove that
$\int_{W} ( D_{h}^{*}.F )(w) \,\mu ( \mathrm{d}w ) = 0$
for $F$ of the form
$
F(w) = f( \int_{0}^{1} \dot{k}_{1}(s) \,\mathrm{d}w(s), \ldots , \int_{0}^{1} \dot{k}_{n}(s) \,\mathrm{d}w(s))
$,
where $k_{1}, \ldots , k_{n} \in H$ are linearly independent. This can be shown by reducing the problem to a multiple integral on Euclidean space, using the fact that
$
(\int_{0}^{1} \dot{k}_{1}(s) \,\mathrm{d}w(s), \ldots , \int_{0}^{1} \dot{k}_{n}(s) \,\mathrm{d}w(s))
$
follows an $n$-dimensional Gaussian distribution.
\end{proof} 

Let us first explain how the more fundamental Cameron--Martin formula is recovered, rather than the Ramer--Kusuoka formula.

\medskip

From now on, fix arbitrarily $h \in H$, which can be expressed as a linear combination of $\{ h \}_{i}$, and introduce an additional structure $Z = \{ Z_{i} \}_{i\in\mathbb{N}}$ in the Heisenberg algebra by setting $Z_{i} = \langle h, h_{i} \rangle_{H}$. This $Z$ collects the Fourier coefficients of $h$, and since it essentially represents $h$ itself, we will henceforth identify $Z$ with $h$. In this case, since
$
[ \langle h, h_{i} \rangle_{H}, D_{j} ]
= [ \langle h, h_{i} \rangle_{H}, D_{j}^{*} ]
= 0
$,
(where $\langle h, h_{i} \rangle_{H}$ is treated as a scalar multiplication operator), the normal order product defined by $\{ D_{i}, D_{i}^{*}, \langle h, h_{i} \rangle_{H} \}_{i}$
(see Section~\ref{Sec:det_order})
satisfies
$
\Phi_{h}
= ( D_{i}^{*} \langle h, h_{j} \rangle_{H} )_{i,j}
= ( \langle h, h_{j} \rangle_{H} D_{i}^{*} )_{i,j}
$,
$
\Psi_{h}
= ( [\, D_{i}^{*}, \langle h, h_{j} \rangle_{H} \,] )_{i,j}
= O
$.
Hence,
\begin{equation*}
\begin{split}
\mathcal{T}\!\exp
\left(
    \sum_{n=0}^{\infty}
    [\, \mathrm{Tr}( \Phi_{h} ( - \Psi_{h})^{n} ) \,]
    \,
    \frac{ t^{n+1} }{ n+1 }
\right)
=
\mathrm{e}^{ t \,\mathrm{Tr}( \Phi_{h} ) }
=
\mathrm{e}^{
    t
    \sum_{i}
    \langle h, h_{i} \rangle_{H} D_{i}^{*}
}
=
\mathrm{e}^{ t D_{h}^{*} }
\end{split}
\end{equation*}
Therefore, by Theorem~\ref{Thm:main},
\begin{equation}
\label{eq:Alg_CM}
\mathrm{e}^{ t D_{h}^{*} }
=
\mathcal{T}\!\exp
\left(
    \sum_{n=0}^{\infty}
    [\, \mathrm{Tr}( \Phi_{h} ( - \Psi_{h})^{n} ) \,]
    \,
    \frac{ t^{n+1} }{ n+1 }
\right)
\stackrel{\text{Theorem~\ref{Thm:main}}}{=}
{\color{light-gray}
\underbrace{
    {\color{black}
    \det ( 1 + t \Psi_{h} )
    }
}_{\text{{\scriptsize $\displaystyle
    \phantom{1}
    =
    1
$}}}
}
\cdot \mathcal{E}_{h}(t)
\cdot T_{-h}(t)
\end{equation}
Let us now compute explicitly the operators appearing on the right-hand side.
First, note that under our representation,
$
D_{h} = \sum_{i} \langle h, h_{i} \rangle_{H} D_{i}
$
(which is consistent with (\ref{eq:def_D})),
and
$
T_{h}(t) \;=\; :\mathrm{e}^{t D_{h}}: \;=\; \mathrm{e}^{t D_{h}}
$,
$
:\mathrm{e}^{t D_{h}^{*}}: \;=\; \mathrm{e}^{t D_{h}^{*}}
$.
Moreover,
\begin{equation*}
\mathrm{e}^{ t ( D_{h} + D_{h}^{*} ) }
\;=\;
\exp \left(
    t
    \sum_{i}
    \langle h, h_{i} \rangle_{H} \int_{0}^{1} \dot{h}_{i}(s) \,\mathrm{d}w(s)
\right)
=
\exp \left(
    t \int_{0}^{1} \dot{h}(s) \,\mathrm{d}w(s)
\right)
\end{equation*}
Furthermore, since the commutator satisfies
$
[ D_{h}, D_{h}^{*} ] = \Vert h \Vert_{2}^{2}
$
(a scalar),
the Baker--Campbell--Hausdorff formula yields
$
\mathrm{e}^{t D_{h}^{*}} \mathrm{e}^{t D_{h} }
=
\mathrm{e}^{ t ( D_{h}^{*} + D_{h} ) + \frac{t^{2}}{2} [ D_{h}^{*}, D_{h} ]  }
$.
Using also that
$D_{h}+D_{h}^{*}=\int_{0}^{1} \dot{h}(s) \,\mathrm{d}w(s)$
is a scalar operator, we obtain
\begin{equation*}
\begin{split}
\mathcal{E}_{h}(t)
\;=\;
:\mathrm{e}^{ t ( D_{h} + D_{h}^{*} ) }:
&\;=\;
\hspace{-23mm}
{\color{light-gray}
\underbrace{
    {\color{black}
    \mathrm{e}^{ t D_{h}^{*} }
    \;
    \mathrm{e}^{ t D_{h} }
    }
}_{\text{{\scriptsize $\displaystyle
    \begin{array}{c}
    \displaystyle
    \phantom{
        \mathrm{e}^{ t ( D_{h}^{*} + D_{h} ) + \frac{t^{2}}{2} [ D_{h}^{*}, D_{h} ]  }
    }
    =
    \mathrm{e}^{ t ( D_{h}^{*} + D_{h} ) + \frac{t^{2}}{2} [ D_{h}^{*}, D_{h} ]  } \\
    \displaystyle
    \phantom{
        \mathrm{e}^{ t ( D_{h} + D_{h}^{*} ) - \frac{t^{2}}{2} \Vert h \Vert_{H}^{2}  }
    }
    =
    \mathrm{e}^{ t ( D_{h} + D_{h}^{*} )
    -
    \frac{t^{2}}{2} \Vert h \Vert_{H}^{2}  } \\
    \end{array}
$}}}
} \\
&\;=\;
\mathrm{e}^{ t ( D_{h} + D_{h}^{*} ) }
\;
\mathrm{e}^{ -\frac{t^{2}}{2} \Vert h \Vert_{H}^{2} }
=
\exp \left(
    t \int_{0}^{1} \dot{h}(s) \,\mathrm{d}w(s)
    -
    \frac{t^{2}}{2}
    \int_{0}^{1} \dot{h}(s)^{2} \,\mathrm{d}s
\right)
\end{split}
\end{equation*}

Finally, for $T_{-h}(t)$, we can compute it as follows.

\begin{Prop} 
\label{Prop:pr-trans}
For any $h \in H$ and $F \in \Gamma^{\mathrm{alg}} ( \int \dot{h} \,\mathrm{d}w )$, we have
\begin{equation*}
( T_{h}(t).F )(w) = F( w + t \cdot h ).
\end{equation*}
\end{Prop} 
\begin{proof} 

By Lemma~\ref{Lem:LEI}, for $F,G \in \Gamma^{\mathrm{alg}}(\int \dot{H} \mathrm{d}w)$, we have
$
T_{h}(t).(FG)
=
( T_{h}(t).F )
\cdot
( T_{h}(t).G )
$.
Therefore, it suffices to prove the claim in the case where
$
F (w) = f( \int_{0}^{1}\dot{k}(s) \,\mathrm{d}w(s) )
$
for some polynomial function $f(x)$ and $k \in H$. This can be verified as follows:
\begin{equation*}
\begin{split}
(T_{h}(t).F)(w)
&=
\sum_{n=0}^{\infty}
\frac{t^{n}}{n!}
D_{h}^{n}.\left( f( \int _{0}^{1} \dot{k}(s) \,\mathrm{d}w(s) ) \right) \\
&=
\sum_{n=0}^{\infty}
\frac{t^{n}}{n!} f^{(n)} ( \int_{0}^{1} \dot{k}(s) \,\mathrm{d}w(s) )
\cdot
\big(
    \hspace{-57mm}
    {\color{light-gray}
    \underbrace{
        {\color{black}
        \langle h, k \rangle_{H}
        }
    }_{\text{{\scriptsize $\displaystyle
        \phantom{
            \left( \int_{0}^{1} \dot{k}(s) \,\mathrm{d}w(s) + \langle h, k \rangle_{H} \right) - \int_{0}^{1} \dot{k}(s) \,\mathrm{d}w(s)
        }
        =
        \left(
        \int_{0}^{1} \dot{k}(s) \,\mathrm{d}w(s)
        +
        \langle h, k \rangle_{H}
        \right)
        -
        \int_{0}^{1} \dot{k}(s) \,\mathrm{d}w(s)
    $}}}
    }
    \hspace{-57mm}
\big)^{n}
\hspace{57mm} \\
&=
f( \int _{0}^{1} \dot{k}(s) \,\mathrm{d}w(s) + \langle k, h \rangle_{H} )
=
F( w + h ).
\end{split}
\end{equation*}
\end{proof} 

Applying both sides of (\ref{eq:Alg_CM}) to
$
F \in \Gamma^{\mathrm{alg}}( \int \dot{H} \,\mathrm{d}w )
$,
we obtain
\begin{equation*}
\begin{split}
(\mathrm{e}^{ t D_{h}^{*} }.F) (w)
&=
(\mathcal{E}_{h}(t) T_{-h}(t).F)(w) \\
&=
\exp \left(
    t \int_{0}^{1} \dot{h}(s) \,\mathrm{d}w(s)
    -
    \frac{t^{2}}{2}
    \int_{0}^{1} \dot{h}(s)^{2} \,\mathrm{d}s
\right)
\cdot
F( w - t \cdot h )
\end{split}
\end{equation*}
Taking expectations on both sides, we recover the Cameron--Martin formula.

\medskip

Next, we explain how the Ramer--Kusuoka formula is recovered.

\medskip

Fix
$
Z = \{ Z_{i} \}_{i} \subset \Gamma^{\mathrm{alg}} ( \int \dot{H} \,\mathrm{d}w)
$
such that $Z_{i}=0$ for all but finitely many $i$, and define a measurable process
$Z(s,w) = Z(s)(w)$
by
\begin{equation*}
Z(s,w) = \sum_{i} Z_{i}(w) h_{i}(s),
\quad
0 \leq s \leq 1,
\quad
w \in W.
\end{equation*}
With respect to the fixed orthonormal basis $\{ h_{i} \}_{i}$, the data $Z$ and $\{ Z(s) \}_{s\in [0,1]}$ are equivalent, and hence we shall identify them in what follows.

The algebra $\mathcal{D}(R)$ is represented as a subalgebra of
$
\mathrm{End}_{\mathbb{R}}( \Gamma^{\mathrm{alg}} ( \int \dot{H} \,\mathrm{d}w ) )
$
generated by $\{ D_{i}, D_{i}^{*}, Z_{i} \}_{i}$, and we consider the corresponding normal ordered product.

As a function of $w$, $Z$ defines a mapping $W \ni w \mapsto (Z(s,w))_{0 \leq s \leq 1} \in H$, and by differentiating this mapping, we can define, for almost every $w \in W$, a trace-class operator 
$
DZ(w) : H \to H
$
by
\begin{equation*}
\langle DZ(w)(h), k \rangle_{H}
=
D_{k}\langle Z(w), h \rangle_{H},
\quad
h,k \in H.
\end{equation*}
Note that the matrix representation of this linear transformation with respect to $\{ h_{i} \}_{i}$ is given by
$\Psi_{Z} = ( [D_{i}^{*},Z_{j}] )_{i,j}$.

Now, our Theorem~\ref{Thm:main} asserts the following:
\begin{equation}
\label{eq:Alg_RK}
\mathcal{T}\!\exp
\left(
    \sum_{n=0}^{\infty}
    [\, \mathrm{Tr}( \Phi_{Z} ( - \Psi_{Z})^{n} ) \,]
    \,
    \frac{ t^{n+1} }{ n+1 }
\right)
\stackrel{\text{Theorem~\ref{Thm:main}}}{=}
\det ( 1 + t \Psi_{Z} )
\cdot \mathcal{E}_{Z}(t)
\cdot T_{-Z}(t)
\end{equation}
The action of $T_{-Z}(t)$ on the right-hand side is, similarly to Proposition~\ref{Prop:pr-trans}, given by
$(T_{-Z}(t).F)(w) = F( w - Z(w) )$.
We now proceed to determine explicitly the remaining multiplicative operator $\mathcal{E}_{Z}(t)$.

By Theorem~\ref{Thm:MB}, we have
\begin{equation*}
\begin{split}
\mathcal{E}_{Z}(t).1(w)
&=
\exp \left( \int_{0}^{t} :\mathrm{e}^{ -t^{\prime} D_{Z} } D_{Z}^{*} : (1) \,\mathrm{d}t^{\prime} \right) (w)
\\
&=
\exp \left(
\int_{0}^{t}
\sum_{i}
\Big( Z_{i} : \mathrm{e}^{-t^{\prime} D_{Z} } : D_{i}^{*} (1) \Big)
\, \mathrm{d}t^{\prime}
\right) (w) \\
&=
\exp \Bigg(
    \int_{0}^{t}
    \sum_{i} Z_{i}
    \sum_{n=0}^{\infty}
    \frac{ (-t^{\prime})^{n} }{ n! }
    \sum_{j_{1}, \ldots, j_{n}}
        Z_{j_{1}} \cdots Z_{j_{n}}
    {\color{light-gray}
    \underbrace{
        {\color{black}
        D_{j_{1}} \cdots D_{j_{n}}
        D_{i}^{*} (1)
        }
    }_{\text{{\scriptsize vanishes for $n\geq 2$}}}
    }
    \,\mathrm{d}t^{\prime}
\Bigg) (w) \\
&=
\exp \Bigg(
    \int_{0}^{t}
    \sum_{i} Z_{i}(w)
    \Bigg(
    D_{i}^{*} (1)
    -
    t^{\prime}
    \;
    {\color{light-gray}
    \underbrace{
        {\color{black}
        \sum_{j}
        Z_{j}(w)
        {\color{light-gray}
        \underbrace{
            {\color{black}
            D_{j}
            D_{i}^{*} (1)
            }
        }_{\text{{\scriptsize $\displaystyle
            \phantom{\delta_{i,j}}
            =
            \delta_{i,j}
        $}}}
        }
        }
    }_{\text{{\scriptsize $\displaystyle
        \phantom{Z_{i}(w)}
        =
        Z_{i}(w)
    $}}}
    }
    \Bigg)
    \,\mathrm{d}t^{\prime}
\Bigg) \\
&=
\exp \Bigg(
    \int_{0}^{t}
    \sum_{i} Z_{i}(w) D_{i}^{*}(1)
    -
    t^{\prime} \sum_{i} Z_{i}(w)^{2}
    \,\mathrm{d}t^{\prime}
\Bigg) \\
&=
\exp \Bigg(
    t
    \sum_{i} Z_{i}(w)
    \hspace{-6mm}
    {\color{light-gray}
    \underbrace{
        {\color{black}
        D_{i}^{*}(1)
        }
    }_{\text{{\scriptsize $\displaystyle
        \begin{array}{c}
        \rotatebox{90}{$=$} \\
        \displaystyle
        \int_{0}^{1} \dot{h}_{i}(s) \,\mathrm{d}w(s)
        \end{array}
    $}}}
    }
    \hspace{-6mm}
    -
    \frac{t^{2}}{2}
    {\color{light-gray}
    \underbrace{
        {\color{black}
        \sum_{i} Z_{i}(w)^{2}
        }
    }_{\text{{\scriptsize $\displaystyle
        \begin{array}{c}
        \rotatebox{90}{$=$} \\
        \displaystyle
        \int_{0}^{1} \dot{Z}(s)^{2} \,\mathrm{d}s
        \end{array}
    $}}}
    }
\Bigg) \\
&=
\exp \Bigg(
    t
    \sum_{i} Z_{i}(w)
    \int_{0}^{1} \dot{h}_{i}(s) \,\mathrm{d}w(s)
    -
    \frac{t^{2}}{2}
    \int_{0}^{1} \dot{Z}(s)^{2} \,\mathrm{d}s
\Bigg)
\end{split}
\end{equation*}
Next, using the integration-by-parts formula with the Malliavin derivative $D_{s}$, $0 \leq s \leq 1$
(e.g. see \cite[p.~36]{NOP}),
we have
\begin{equation*}
Z_{i}(w)
\int_{0}^{1} \dot{h}_{i}(s) \,\mathrm{d}w(s)
=
\int_{0}^{1} Z_{i}(w) \dot{h}_{i}(s) \,\delta w(s)
+
\int_{0}^{1} \dot{h}_{i}(s) D_{s}Z_{i}(w) \mathrm{d}s.
\end{equation*}
Since
$
\int_{0}^{1} \dot{h}_{i}(s) D_{s}Z_{i}(w) \,\mathrm{d}s
=
\langle D_{i}Z, h_{i} \rangle_{H} = D_{i}Z_{i}
$,
we obtain
\begin{equation*}
\sum_{i}
Z_{i}(w)
\int_{0}^{1} \dot{h}_{i}(s) \,\mathrm{d}w(s)
=
\int_{0}^{1} \dot{Z}(s) \,\delta w(s)
+
\mathrm{tr}( DZ(w) )
\end{equation*}
Hence,
\begin{equation*}
\begin{split}
\mathcal{E}_{Z}(t).1(w)
=
\mathrm{e}^{ t \cdot \mathrm{tr}( DZ(w) ) }
\cdot
\exp \Bigg(
    t
    \int_{0}^{1} \dot{Z}(s) \,\delta w(s)
    -
    \frac{t^{2}}{2}
    \int_{0}^{1} \dot{Z}(s)^{2} \,\mathrm{d}s
\Bigg)
\end{split}
\end{equation*}

Therefore, applying both sides of (\ref{eq:Alg_RK}) to
$
F \in \Gamma^{\mathrm{alg}}( \int \dot{H} \,\mathrm{d}s )
$,
we obtain
\begin{equation*}
\begin{split}
&
\mathcal{T}\!\exp
\left(
    \sum_{n=0}^{\infty}
    [\, \mathrm{Tr}( \Phi_{Z} ( - \Psi_{Z})^{n} ) \,]
    \,
    \frac{ t^{n+1} }{ n+1 }
\right).F (w) \\
&=
\det ( 1 - t DZ(w) )
\cdot
\mathrm{e}^{ t \cdot \mathrm{tr}( DZ(w) ) }
\cdot
\exp \Bigg(
    t
    \int_{0}^{1} \dot{Z}(s) \,\delta w(s)
    -
    \frac{t^{2}}{2}
    \int_{0}^{1} \dot{Z}(s)^{2} \,\mathrm{d}s
\Bigg) \\
&\hspace{20mm} \times
F( w - t \cdot Z(w) )
\end{split}
\end{equation*}
Taking expectations of both sides, we recover the Ramer--Kusuoka formula.

\begin{Rm} 
In the present framework, we have used a setting in which the normal ordered product and time ordering are independent. However, \cite{AAU} develops an interesting framework in which these structures are intertwined.
\end{Rm} 


\end{document}